\documentclass{amsart}
\usepackage{verbatim}
\usepackage{amssymb}
\usepackage{amsmath}
\usepackage[usenames]{color}
\usepackage{graphicx}
\usepackage{amscd}
\usepackage[numbers,sort&compress,square]{natbib}

\newtheorem{theorem}{Theorem}
\numberwithin{theorem}{section}

\newtheorem{proposition}[theorem]{Proposition}
\newtheorem{proposition*}[theorem]{Proposition$^*$}
\newtheorem{lemma}[theorem]{Lemma}
\newtheorem{lemma*}[theorem]{Lemma$^*$}
\newtheorem{theorem*}[theorem]{Theorem$^*$}
\newtheorem{corollary}[theorem]{Corollary}

\newtheorem{definition}[theorem]{Definition}

\newtheorem{remark}[theorem]{Remark}

\numberwithin{equation}{section}
\numberwithin{theorem}{section}

\newtheorem{innercustomthm}{Theorem}
\newenvironment{customthm}[1]
  {\renewcommand\theinnercustomthm{#1}\innercustomthm}
  {\endinnercustomthm}
  
\newtheorem{innercustomprop}{Proposition}
\newenvironment{customprop}[1]
  {\renewcommand\theinnercustomprop{#1}\innercustomprop}
  {\endinnercustomprop}

\newtheorem{innercustomlemma}{Lemma}
\newenvironment{customlemma}[1]
  {\renewcommand\theinnercustomlemma{#1}\innercustomlemma}
  {\endinnercustomlemma}
\allowdisplaybreaks

\newcommand{\eps}{\varepsilon}

\newcommand{\sign}{\mathrm{sign}\text{ }}

\newcommand{\ud}[0]{\,\mathrm{d}}

\newcommand{\vertiii}[1]{{\left\vert\kern-0.25ex\left\vert\kern-0.25ex\left\vert #1 
    \right\vert\kern-0.25ex\right\vert\kern-0.25ex\right\vert}}

\begin{document}

\title[UMD spaces and martingale decompositions]
{Martingale decompositions and weak differential subordination 
 in UMD Banach spaces}

\author{Ivan Yaroslavtsev}
\address{Delft Institute of Applied Mathematics\\
Delft University of Technology \\ P.O. Box 5031\\ 2600 GA Delft\\The
Netherlands}
\email{I.S.Yaroslavtsev@tudelft.nl}

\begin{abstract}
In this paper we consider Meyer-Yoeurp decompositions for UMD Banach space-valued martingales. Namely, we prove that $X$ is a UMD Banach space if and only if for any fixed $p\in (1,\infty)$, any $X$-valued $L^p$-martingale $M$ has a~unique decomposition $M = M^d + M^c$ such that $M^d$ is a~purely discontinuous martingale, $M^c$ is a continuous martingale, $M^c_0=0$ and 
\[
 \mathbb E \|M^d_{\infty}\|^p + \mathbb E \|M^c_{\infty}\|^p\leq c_{p,X} \mathbb E \|M_{\infty}\|^p.
\]
An analogous assertion is shown for the Yoeurp decomposition of a purely discontinuous martingales into a sum of a~quasi-left continuous martingale and a martingale with accessible jumps. 

As an application we show that $X$ is a UMD Banach space if and only if for any fixed $p\in (1,\infty)$ and for all $X$-valued martingales $M$ and $N$ such that $N$ is weakly differentially subordinated to $M$, one has the estimate 
$$
\mathbb E \|N_{\infty}\|^p \leq C_{p,X}\mathbb E \|M_{\infty}\|^p.
$$
\end{abstract}

\keywords{differential subordination, weak differential subordination, UMD Banach spaces, Burkholder function, stochastic integration, Brownian representation, Meyer-Yoeurp decomposition, Yoeurp decomposition, purely discontinuous martingales, continuous martingales, quasi-left continuous, accessible jumps, canonical decomposition of martingales}

\subjclass[2010]{60G44 Secondary: 60B11, 60G46}

\maketitle

\section{Introduction}

It is well-known from the fundamental paper of It\^o \cite{Ito42} on the real-valued case, and several works \cite{AlRud05,App07,RieGaa,Dett82,Baum15} on the vector-valued case, that for any Banach space~$X$, any centered $X$-valued L\'evy process has a unique decomposition \mbox{$L = W +\widetilde N$}, where $W$ is an $X$-valued Wiener process, and $\widetilde N$ is an $X$-valued weak integral with respect to a certain compensated Poisson random measure. Moreover, $W$ and $\widetilde N$ are independent, and therefore since $W$ is symmetric, for each $1<p<\infty$ and $t\geq 0$,
\begin{equation}\label{eq:LevyItop-norm}
 \mathbb E \|\widetilde N_t\|^p \leq \mathbb E \|L_t\|^p. 
\end{equation}

The natural generalization of this result to general martingales in the real-valued setting was provided by Meyer in \cite{Mey76} and Yoeurp in \cite{Yoe76}. Namely, it was shown that any real-valued martingale $M$ can be uniquely decomposed into a sum of two martingales $M^d$ and $M^c$ such that $M^d$ is purely discontinuous (i.e.\ the quadratic variation $[M^d]$ has a pure jump version), and $M^c$ is continuous with $M^c_0=0$. The reason why they needed such a decomposition is a further decomposition of a semimartingale, and finding an exponent of a semimartingale (we refer the reader to \cite{Kal} and \cite{Yoe76} for the details on this approach). In the present article we extend Meyer-Yoeurp theorem to the vector-valued setting, and provide extension of \eqref{eq:LevyItop-norm} for a~general martingale (see Subsection \ref{subsec:MeyYoudec}). Namely, we prove that for any UMD Banach space $X$ and any $1<p<\infty$, an $X$-valued $L^p$-martingale $M$ can be uniquely decomposed into a~sum of two martingales $M^d$ and $M^c$ such that $M^d$ is purely discontinuous (i.e.\ $\langle M^d, x^*\rangle$ is purely discontinuous for each $x^* \in X^*$), and $M^c$ is continuous with $M^c_0=0$. Moreover, then for each $t\geq 0$,
\begin{equation}\label{eq:MeyYoeIntro}
  (\mathbb E \|M^d_t\|^p)^{\frac 1p} \leq \beta_{p, X}(\mathbb E \|M_t\|^p)^{\frac 1p},\;\;\;\;\;
   (\mathbb E \|M^c_t\|^p)^{\frac 1p} \leq \beta_{p, X}(\mathbb E \|M_t\|^p)^{\frac 1p},
\end{equation}
where $\beta_{p, X}$ is the UMD$_p$ constant of $X$ (see Subsection \ref{subsec:UMD}). Theorem \ref{thm:exampleforlowboundofA^qA^a} shows that such a decomposition together with $L^p$-estimates of type \eqref{eq:MeyYoeIntro} is possible if and only if $X$ has the UMD property.

The purely discontinuous part can be further decomposed: in \cite{Yoe76} Yoeurp proved that any real-valued purely discontinuous $M^d$ can be uniquely decomposed into a sum of a purely discontinuous quasi-left continuous martingale~$M^q$ (analogous to the ``compensated Poisson part'', which does not jump at predictable stopping times), and a purely discontinuous martingale with accessible jumps $M^a$ (analogous to the ``discrete part'', which jumps only at certain predictable stopping times). In Subsection \ref{subsec:Youdec} we extend this result to a UMD space-valued setting with appropriate estimates. Namely, we prove that for each $1<p<\infty$ the same type of decomposition is possible and unique for an $X$-valued purely discontinuous $L^p$-martingale $M^d$, and then for each $t\geq 0$,
\begin{equation}\label{eq:YouIntro}
  (\mathbb E \|M^q_t\|^p)^{\frac 1p} \leq \beta_{p, X}(\mathbb E \|M^d_t\|^p)^{\frac 1p},\;\;\;\;\;
   (\mathbb E \|M^a_t\|^p)^{\frac 1p}\leq \beta_{p, X}(\mathbb E \|M^d_t\|^p)^{\frac 1p}.
\end{equation}
Again as Theorem \ref{thm:exampleforlowboundofA^qA^a} shows, the \eqref{eq:YouIntro}-type estimates are a possible only in UMD Banach spaces.

\smallskip

Even though the Meyer-Yoeurp and Yoeurp decompositions can be easily extended from the real-valued case to a Hilbert space case, the author could not find the corresponding estimates of type \eqref{eq:MeyYoeIntro}-\eqref{eq:YouIntro} in the literature, so we wish to present this special issue here. If $H$ is a Hilbert space, $M:\mathbb R_+ \times \Omega \to H$ is a martingale, then there exists a unique decomposition of $M$ into a continuous part $M^c$, a purely discontinuous quasi-left continuous part $M^q$, and a purely discontinuous part $M^a$ with accessible jumps. Moreover, then for each $1<p<\infty$, and for $i=c,q,a$,
\begin{equation}\label{eq:Hvaluedcasintro}
  (\mathbb E \|M^i_t\|^p)^{\frac 1p} \leq (p^*-1)(\mathbb E \|M_t\|^p)^{\frac 1p},
\end{equation}
where $p^* = \max\{p, \frac{p}{p-1}\}$. Notice that though \eqref{eq:Hvaluedcasintro} follows from \eqref{eq:MeyYoeIntro}-\eqref{eq:YouIntro} since $\beta_{p,H}=p^*-1$, it can be easily derived from the differential subordination estimates for Hilbert space-valued martingales obtained by Wang in~\cite{Wang}.

\smallskip

Both the Meyer-Yoeurp and Yoeurp decompositions play a significant r\^ole in stochastic integration: if $M=M^c + M^q + M^a$ is a decomposition of an $H$-valued martingale $M$ into continuous, purely discontinuous quasi-left continuous and purely discontinuous with accessible jumps parts, and if $\Phi:\mathbb R_+ \times \Omega \to \mathcal L(H, X)$ is elementary predictable for some UMD Banach space $X$, then the decomposition $\Phi \cdot M=\Phi \cdot M^c + \Phi \cdot M^q + \Phi \cdot M^a$ of a stochastic integral $\Phi \cdot M$ is a decomposition of the martingale $\Phi \cdot M$ into continuous, purely discontinuous quasi-left continuous and purely discontinuous with accessible jumps parts, and for any $1<p<\infty$ we have that
\[
 \mathbb E \|(\Phi \cdot M)_{\infty}\|^p \eqsim_{p,X}\mathbb E \|(\Phi \cdot M^c)_{\infty}\|^p + \mathbb E \|(\Phi \cdot M^q)_{\infty}\|^p + \mathbb E \|(\Phi \cdot M^a)_{\infty}\|^p.
\]
The corresponding It\^o isomorphism for $\Phi \cdot M^c$ for a general UMD Banach space $X$ was derived by Veraar and the author in \cite{VY2016}, while It\^o isomorphisms for $\Phi \cdot M^q$ and $\Phi \cdot M^a$ have been shown by Dirksen and the author in \cite{DY17} for the case $X = L^r(S)$, $1<r<\infty$.

\smallskip

The major underlying techniques involved in the proofs of \eqref{eq:MeyYoeIntro} and \eqref{eq:YouIntro} are rather different from the original methods of Meyer in \cite{Mey76} and Yoeurp in \cite{Yoe76}. They include the results on the differentiability of the Burkholder function of any finite dimensional Banach space, which have been proven recently in \cite{Y17FourUMD} and which allow us to use It\^o formula in order to show the desired inequalities in the same way as it was demonstrated by Wang in \cite{Wang}.

\smallskip

The main application of the Meyer-Yoeurp decomposition are $L^p$-estimates for weakly differentially subordinated martingales. The weak differential subordination property was introduced by the author in \cite{Y17FourUMD}, and can be described in the following way: an $X$-valued martingale $N$ is weakly differentially subordinated to an $X$-valued martingale $M$ if for each $x^* \in X^*$  a.s.\ $|\langle N_0, x^*\rangle| \leq |\langle M_0, x^*\rangle|$ and for each $t\geq s\geq 0$
\[
 [\langle N, x^*\rangle]_t - [\langle N, x^*\rangle]_s \leq [\langle M, x^*\rangle]_t - [\langle M, x^*\rangle]_s.
\]
If both $M$ and $N$ are purely discontinuous, and if $X$ is a UMD Banach space, then by \cite{Y17FourUMD}, for each $1<p<\infty$ we have that $\mathbb E \|N_{\infty}\|^p \leq \beta_{p, X}^p\mathbb E \|M_{\infty}\|^p$. Section \ref{sec:weakdifsubandgenmart} is devoted to the generalization of this result to continuous and general martingales. There we show that if both $M$ and $N$ are continuous, then $\mathbb E \|N_{\infty}\|^p \leq c_{p,X}^p\mathbb E \|M_{\infty}\|^p$, where the least admissible $c_{p,X}$ is within the interval $[\beta_{p, X}, \beta_{p, X}^2]$. Furthermore, using the Meyer-Yoeurp decomposition and estimates \eqref{eq:MeyYoeIntro} we show that for general \mbox{$X$-va}lu\-ed martingales $M$ and $N$ such that $N$ is weakly differentially subordinated to $M$ the following holds
\[
 (\mathbb E \|N_{\infty}\|^p)^{\frac 1p} \leq \beta_{p, X}^2(\beta_{p, X}+1) (\mathbb E \|M_{\infty}\|^p)^{\frac 1p}. 
\]

\smallskip

The weak differential subordination as a stronger version of the differential subordination is of interest in Harmonic Analysis. For instance, it was shown in \cite{Y17FourUMD} that sharp $L^p$-estimates for weakly differentially subordinated purely discontinuous martingales imply sharp estimates for the norms of a broad class of Fourier multipliers on $L^p(\mathbb R^d;X)$. Also there is a strong connection between the weak differential subordination of continuous martingales and the norm of the~Hilbert transform on $L^p(\mathbb R; X)$ (see \cite{Y17FourUMD} and Remark~\ref{rem:Hiltranfweakdiffsubcont}).
 
Alternative approaches to Fourier multipliers for functions with values in UMD spaces have been constructed from the differential subordination for purely discontinuous martingales (see Ba\~{n}uelos and Bogdan \cite{BB}, Ba\~{n}uelos, Bogdan and Bie\-la\-szew\-ski \cite{BBB}, and recent work \cite{Y17FourUMD}), and for continuous martingales (see McConnell \cite{McC84} and Geiss, Montgomery-Smith and Saksman \cite{GM-SS}). It remains open whether one can combine these two approaches using the general weak differential subordination theory.

\section{Preliminaries}
In the sequel we will omit proofs of some statements marked with a star (e.g.\ Lemma$^*$, Theorem$^*$, etc.) Please find the corresponding proofs after the references or in the supplement \cite{Y17MDSupp}.

\smallskip

We set the scalar field to be $\mathbb R$. We will use the {\em Kronecker symbol} $\delta_{ij}$, which is defined in the following way: $\delta_{ij} =1$ if $i=j$, and $\delta_{ij} =0$ if $i\neq j$. For each $p\in (1,\infty)$ we set $p'\in (1,\infty)$ and $p^* \in [2,\infty)$ to be such that $\frac 1p + \frac 1{p'}=1$ and $p^* = \max\{p, p'\}$. We set $\mathbb R_+ := [0,\infty)$.

\subsection{UMD Banach spaces}\label{subsec:UMD}\nopagebreak
A Banach space $X$ is called a {\em UMD space} if for some (equivalently, for all)
$p \in (1,\infty)$ there exists a constant $\beta>0$ such that
for every $n \geq 1$, every martingale
difference sequence $(d_j)^n_{j=1}$ in $L^p(\Omega; X)$, and every $\{-1,1\}$-valued sequence
$(\varepsilon_j)^n_{j=1}$
we have
\[
\Bigl(\mathbb E \Bigl\| \sum^n_{j=1} \varepsilon_j d_j\Bigr\|^p\Bigr )^{\frac 1p}
\leq \beta \Bigl(\mathbb E \Bigl \| \sum^n_{j=1}d_j\Bigr\|^p\Bigr )^{\frac 1p}.
\]
The least admissible constant $\beta$ is denoted by $\beta_{p,X}$ and is called the {\em UMD constant}. It is well-known (see \cite[Chapter 4]{HNVW1}) that $\beta_{p, X}\geq p^*-1$ and that $\beta_{p, H} = p^*-1$ for a Hilbert space $H$. 
 We refer the reader to \cite{Burk01,HNVW1,Rubio86,Pis16} for details.

The following proposition is a vector-valued version of \cite[Theorem 4.1]{Choi92}.

\begin{proposition}\label{prop:UMD01sequence}
 Let $X$ be a Banach space, $p\in (1,\infty)$. Then $X$ has the UMD property if and only if there exists $C>0$ such that for each $n\geq 1$, for every martingale
difference sequence $(d_j)^n_{j=1}$ in $L^p(\Omega; X)$, and every sequence
$(\varepsilon_j)^n_{j=1}$ such that $\varepsilon_j\in \{0,1\}$ for each $j=1,\ldots,n$
we have
\begin{align*}
 \Bigl(\mathbb E \Bigl\| \sum^n_{j=1} \varepsilon_j d_j\Bigr\|^p\Bigr )^{\frac 1p}
\leq C \Bigl(\mathbb E \Bigl \| \sum^n_{j=1}d_j\Bigr\|^p\Bigr )^{\frac 1p}.
\end{align*}
If this is the case, then the least admissible $C$ is in the interval $[\frac {\beta_{p, X}-1}{2}, \beta_{p, X}]$
\end{proposition}

\subsection{Martingales and stopping times in continuous time}\label{subsec:prelimmart}
Let $(\Omega,\mathcal F, \mathbb P)$ be a probability space with a filtration $\mathbb F = (\mathcal F_t)_{t\geq 0}$ which satisfies the usual conditions. Then $\mathbb F$ is right-continuous, and the following proposition holds (see \cite{Y17FourUMD}):
\begin{proposition}\label{prop:cadlagvers}
 Let $X$ be a Banach space. Then any martingale $M:\mathbb R_+ \times \Omega \to X$ has a {\em c\`adl\`ag} version
\end{proposition}

Let $1\leq p\leq \infty$. A martingale $M:\mathbb R_+ \times \Omega \to X$ is called an {\em $L^p$-martingale} if $M_t \in L^p(\Omega; X)$ for each $t\geq 0$, there exists an a.s.\ limit $M_{\infty} := \lim_{t\to \infty} M_t$, $M_{\infty} \in L^p(\Omega; X)$ and $M_t \to M_{\infty}$ in $L^p(\Omega; X)$ as $t\to \infty$. We will denote the space of all $X$-valued $L^p$-martingales on $\Omega$ by $\mathcal M_X^p(\Omega)$. For brevity we will use $\mathcal M_X^p$ instead. Notice that $\mathcal M_X^p$ is a Banach space with the given norm: $\|M\|_{\mathcal M_X^p} := \|M_{\infty}\|_{L^p(\Omega; X)}$ (see \cite{Kal,Jac79} and \cite[Chapter 1]{HNVW1}).

\begin{proposition}\label{prop:dualofMXp}
  Let $X$ be a Banach space with the Radon-Nikod\'ym property (e.g.\ reflexive), $1<p<\infty$. Then $(\mathcal M_X^p)^* = \mathcal M_{X^*}^{p'}$, and $\|M\|_{(\mathcal M_X^p)^*} = \|M\|_{\mathcal M_{X^*}^{p'}}$ for each $M \in \mathcal M_{X^*}^{p'}$.
\end{proposition}

A random variable $\tau:\Omega \to \mathbb R_+$ is called an {\em optional stopping time} (or just a {\em stopping time}) if $\{\tau\leq t\} \in \mathcal F_t$ for each $t\geq 0$. With an optional stopping time $\tau$ we associate a $\sigma$-field $\mathcal F_{\tau} = \{A\in \mathcal F_{\infty}: A\cap \{\tau\leq t\}\in \mathcal F_{t}, t\in\mathbb R_+\}$. Note that $M_{\tau}$ is strongly $\mathcal F_{\tau}$-measurable for any local martingale $M$. We refer to \cite[Chapter 7]{Kal} for details.

Due to the existence of a c\`adl\`ag version of a martingale $M:\mathbb R_+ \times \Omega \to X$, we can define an $X$-valued random variables $M_{\tau-}$ and $\Delta M_{\tau}$ for any stopping time $\tau$ in the following way: $M_{\tau-} = \lim_{\eps \to 0}M_{(\tau - \eps)\vee 0}$, $\Delta M_{\tau} = M_{\tau} -  M_{\tau-}$.

\subsection{Quadratic variation}
Let $(\Omega, \mathcal F, \mathbb P)$ be a probability space with a filtration $\mathbb F = (\mathcal F_t)_{t\geq 0}$ that
satisfies the usual conditions, $H$ be a Hilbert space. Let $M:\mathbb R_+ \times \Omega \to H$ be a local martingale. We define a {\em quadratic variation} of $M$ in the following way:
\begin{equation}\label{eq:defquadvar}
 [M]_t  := \mathbb P-\lim_{{\rm mesh}\to 0}\sum_{n=1}^N \|M(t_n)-M(t_{n-1})\|^2,
\end{equation}
where the limit in probability is taken over partitions $0= t_0 < \ldots < t_N = t$. Note that $[M]$ exists and is nondecreasing a.s. The reader can find more on quadratic variations in \cite{MetSemi,MP} for the vector-valued setting, and in \cite{Kal,Prot,MP} for the real-valued setting.

For any martingales $M, N:\mathbb R_+ \times \Omega \to H$ we can define a {\em covariation} $[M,N]:\mathbb R_+ \times \Omega \to \mathbb R$ as $[M,N] := \frac{1}{4}([M+N]-[M-N])$.
Since $M$ and $N$ have c\`adl\`ag versions, $[M,N]$ has a c\`adl\`ag version as well (see \cite[Theorem I.4.47]{JS} and \cite{MetSemi}).
\begin{remark}[\cite{MetSemi}]\label{rem:covariation}
 The process $\langle M,N\rangle - [M,N]$ is a local martingale.
\end{remark}

\subsection{Continuous martingales}\label{subsec:prelimcontmart}
 Let $X$ be a Banach space. A martingale $M:\mathbb R_+ \times \Omega \to X$ is called {\em continuous} if $M$ has continuous paths.

\begin{remark}[\cite{Kal,MP}]\label{rem:qvcont}
 If $X$ is a Hilbert space, $M,N:\mathbb R_+ \times \Omega \to X$ are continuous martingales, then $[M, N]$ has a continuous version.
\end{remark}

Let $1\leq p\leq \infty$. We will denote the linear space of all continuous $X$-valued $L^p$-martingales on $\Omega$ which start at zero by $\mathcal M_X^{p,c}(\Omega)$. For brevity we will write $\mathcal M_X^{p,c}$ instead of $\mathcal M_X^{p,c}(\Omega)$ since $\Omega$ is fixed. Analogously to \cite[Lemma 17.4]{Kal} by applying Doob's maximal inequality \cite[Theorem 3.2.2]{HNVW1} one can show the following proposition.

\begin{proposition}\label{prop:M^p,ciscomplete}
 Let $X$ be a Banach space, $p\in(1,\infty)$. Then $\mathcal M_X^{p,c}$ is a Banach space with the following norm: $\|M\|_{\mathcal M_X^{p,c}} := \|M_{\infty}\|_{L^p(\Omega; X)}$.
\end{proposition}

\subsection{Purely discontinuous martingales}\label{subsec:prelimpdmart}
 An increasing c\`adl\`ag process $A:\mathbb R_+ \times \Omega \to \mathbb R$ is called {\em pure jump} if a.s.\ for each $t\geq 0$, $A_t = A_0 + \sum_{s=0}^t \Delta A_s$.
A local martingale $M:\mathbb R_+\times \Omega \to \mathbb R$ is called {\em purely discontinuous} if $[M]$ is a pure jump process.
The reader can find more on purely discontinuous martingales in \cite{JS,Kal}. We leave the following evident lemma without proof.

\begin{lemma}\label{lemma:A^dA^c}
 Let $A:\mathbb R_+ \times \Omega \to \mathbb R_+$ be an increasing adapted c\`adl\`ag process such that $A_0 =0$. Then there exist unique up to indistinguishability increasing adapted c\`adl\`ag processes $A^c, A^d:\mathbb R_+ \times \Omega \to \mathbb R_+$ such that $A^c$ is continuous a.s., $A^d$ is pure jump a.s., $A^c_0 = A^d_0=0$ and $A= A^c + A^d$.
\end{lemma}

\begin{remark}\label{rem:YoeMeyR}
 According to the works \cite{Mey76} by Meyer and \cite{Yoe76} by Yoeurp (see also \cite[Theorem~26.14]{Kal}), any martingale $M:\mathbb R_+ \times \Omega \to \mathbb R$ can be uniquely decomposed into a sum of a purely discontinuous local martingale $M^d$ and a continuous local martingale $M^c$ such that $M^c_0 =0$. Moreover, $[M]^c= [M^c]$ and $[M]^d = [M^d]$, where $[M]^c$ and $[M]^d$ are defined as in Lemma \ref{lemma:A^dA^c}.
\end{remark}

\begin{corollary}\label{cor:cont=purdiscR}
 Let $M:\mathbb R_+ \times \Omega \to\mathbb R$ be a martingale which is both continuous and purely discontinuous. Then $M=M_0$ a.s.
\end{corollary}

 \begin{proposition*}\label{thm:purdiscorthtoanycont1}
  A martingale $M:\mathbb R_+\times \Omega \to \mathbb R$ is purely discontinuous if and only if $M N$ is a martingale for any continuous bounded martingale $N:\mathbb R_+ \times \Omega \to \mathbb R$ with $N_0 = 0$. 
\end{proposition*}
 
 Note that some authors take this equivalent condition as the definition of a purely discontinuous martingale, see e.g.\ \cite[Definition I.4.11]{JS} and \cite[Chapter I]{Jac79}.

\begin{definition}\label{def:purelydiscXvalued}
 Let $X$ be a Banach space, $M:\mathbb R_+\times \Omega \to X$ be a local martingale. Then $M$ is called {\em purely discontinuous} if for each $x^* \in X^*$ the local martingale $\langle M, x^*\rangle$ is purely discontinuous. 
\end{definition}

\begin{remark}\label{rem:YoeMey}
 Let $X$ be finite dimensional. Then similarly to Remark \ref{rem:YoeMeyR} any martingale $M:\mathbb R_+ \times \Omega \to X$ can be uniquely decomposed into a sum of a purely discontinuous local martingale $M^d$ and a continuous local martingale $M^c$ such that $M^c_0 =0$.
\end{remark}

\begin{remark}\label{rem:purdiscorthtoanycont2}
 Analogously to Proposition \ref{thm:purdiscorthtoanycont1}, a martingale $M:\mathbb R_+\times \Omega \to X$ is purely discontinuous if and only if $\langle M, x^*\rangle N$ is a martingale for any $x^* \in X^*$ and any continuous bounded martingale $N:\mathbb R_+ \times \Omega \to \mathbb R$ such that $N_0 = 0$. 
\end{remark}

Let $p\in [1,\infty]$. We will denote the linear space of all purely discontinuous $X$-valued $L^p$-martingales on $\Omega$ by $\mathcal M_X^{p,d}(\Omega)$. Since $\Omega$ is fixed, we will use $\mathcal M_X^{p,d}$ instead.
The scalar case of the next result have been presented in \cite[Lemme I.2.12]{Jac79}.

\begin{proposition}\label{thm:M^p,discomplete}
 Let $X$ be a Banach space, $p\in (1,\infty)$. Then $\mathcal M_X^{p, d}$ is a Banach space with a norm defined as follows: $\|M\|_{\mathcal M_X^{p, d}}:= \|M_{\infty}\|_{L^p(\Omega; X)}$.
\end{proposition}
\begin{proof}
Let $(M^n)_{n\geq 1}$ be a sequence of purely discontinuous $X$-valued $L^p$-martingales such that $(M^n_{\infty})_{n\geq 1}$ is a Cauchy sequence in $L^p(\Omega; X)$. Let $\xi\in L^p(\Omega; X)$ be such that $\lim_{n\to \infty}M^n_{\infty} =\xi$. Define a martingale $M:\mathbb R_+ \times \Omega \to X$ as follows: $M = (M_s)_{s\geq 0} = (\mathbb E(\xi|\mathcal F_s))_{s\geq 0}$. Let us show that $M \in \mathcal M_X^{p,d}$. First notice that $\|M_{\infty}\|_{L^p(\Omega; X)} = \|\xi\|_{L^p(\Omega; X)} < \infty$. Further for each $x^* \in X^*$ by \cite[Lemme I.2.12]{Jac79} we have that $\langle M, x^*\rangle$ as a limit of real-valued purely discontinuous martingales $(\langle M^n, x^*\rangle)_{n\geq 1}$ in $\mathcal M_{\mathbb R}^p$ is purely discontinuous. Therefore $M$ is purely discontinuous by the definition.
\end{proof}

\begin{lemma}\label{lemma:contpuredisczero}
 Let $X$ be a Banach space, $M:\mathbb R_+ \times \Omega \to X$ be a martingale such that $M$ is both continuous and purely discontinuous. Then $M = M_0$ a.s.
\end{lemma}

\begin{proof}
Follows analogously Corollary \ref{cor:cont=purdiscR}.
\end{proof}

\subsection{Time-change}
A nondecreasing, right-continuous family of stopping times $\tau =(\tau_s)_{s\geq 0}$ is called
a  \textit{random time-change}. If $\mathbb F$ is right-continuous, then
according to \cite[Lemma~7.3]{Kal} the same holds true
for the  \textit{induced filtration} $\mathbb G = (\mathcal G_s)_{s \geq 0} = (\mathcal F_{\tau_s})_{s\geq 0}$
(see more in \cite[Chapter~7]{Kal}).
Let $X$ be a Banach space. A~martingale $M:\mathbb R_+ \times \Omega \to X$ is said to be  \textit{$\tau$-continuous} if $M$ is an a.s.\ constant on every interval $[\tau_{s-}, \tau_s]$, $s \geq 0$, where we let $\tau_{0-} = 0$.

\begin{theorem*}\label{thm:apptimechange}
 Let $A:\mathbb R_+ \times \Omega\to \mathbb R_+$ be a strictly increasing continuous predictable process such that $A_0 = 0$ and $A_{t} \to \infty$ as $t\to \infty$ a.s. Let $\tau = (\tau_s)_{s\geq 0}$ be a random time-change defined as
  $\tau_s := \{t: A_t=s\}$, $s\geq 0$.
Then $(A\circ \tau)(t) = (\tau \circ A)(t) = t$ a.s.\ for each $t\geq 0$. Let $\mathbb G = (\mathcal G_s)_{s \geq 0} = (\mathcal F_{\tau_s})_{s\geq 0}$ be the induced filtration. Then $(A_t)_{t\geq 0}$ is a random time-change with respect to $\mathbb G$ and for any $\mathbb F$-martingale $M:\mathbb R_+ \times \Omega \to\mathbb R$ the following holds
\begin{itemize}
 \item [(i)] $M \circ \tau$ is a continuous $\mathbb G$-martingale if and only if $M$ is continuous, and
 \item[(ii)]$M \circ \tau$ is a purely discontinuous $\mathbb G$-martingale if and only if $M$ is purely discontinuous.
\end{itemize}
\end{theorem*}

\subsection{Stochastic integration}
Let $X$ be a Banach space, $H$ be a Hilbert space. For each $h\in H$, $x\in X$ we denote the linear operator $g\mapsto \langle g, h\rangle x$, $g\in H$, by $h\otimes x$. The process $\Phi: \mathbb R_+ \times \Omega \to \mathcal L(H,X)$ is called  \textit{elementary progressive}
with respect to the filtration $\mathbb F = (\mathcal F_t)_{t \geq 0}$ if it is of the form
\begin{equation}\label{eq:elprog}
 \Phi(t,\omega) = \sum_{k=1}^K\sum_{m=1}^M \mathbf 1_{(t_{k-1},t_k]\times B_{mk}}(t,\omega)
\sum_{n=1}^N h_n \otimes x_{kmn},\;\;\; t\geq 0, \omega \in \Omega,
\end{equation}
where $0 \leq t_0 < \ldots < t_K <\infty$, for each $k = 1,\ldots, K$ the sets
$B_{1k},\ldots,B_{Mk}$ are in $\mathcal F_{t_{k-1}}$ and the vectors $h_1,\ldots,h_N$ are orthogonal.
Let $M:\mathbb R_+ \times \Omega \to H$ be a martingale. Then we define the {\em stochastic integral} $\Phi \cdot M:\mathbb R_+ \times \Omega \to X$ of $\Phi$ with respect to $M$ as follows:
\begin{equation}\label{eq:defofstochintwrtM}
 (\Phi \cdot M)_t = \sum_{k=1}^K\sum_{m=1}^M \mathbf 1_{B_{mk}}
\sum_{n=1}^N \langle(M(t_k\wedge t)- M(t_{k-1}\wedge t)), h_n\rangle x_{kmn},\;\; t\geq 0.
\end{equation}

We will need the following lemma on stochastic integration (see \cite{Y17FourUMD}).

\begin{lemma}\label{lemma:stochintmoment}
 Let $d$ be a natural number, $H$ be a $d$-dimensional Hilbert space, $p\in (1,\infty)$, $M, N:\mathbb R_+ \times \Omega \to H$ be $L^p$-martingales, $F:H \to H$ be a measurable function such that $\|F(h)\|\leq C \|h\|^{p-1}$ for each $h\in H$ and some $C>0$. Define $N_-:\mathbb R_+ \times \Omega \to H$ by $(N_-)_t = N_{t-}$, $t\geq 0$. Then $F(N_{-})\cdot M$ is a martingale and for each $t\geq 0$
 \begin{equation}\label{eq:stochintmoment}
  \mathbb E |(F(N_-)\cdot M)_t|\lesssim_{p, d} C(\mathbb E \|N_t\|^p)^{\frac {p-1}p} (\mathbb E \|M_t\|^p)^{\frac {1}p}. 
 \end{equation}
\end{lemma}

\subsection{Multidimensional Wiener process}
Let $d$ be a natural number. $W:\mathbb R_+ \times \Omega \to \mathbb R^d$ is called a {\em standard $d$-dimensional Wiener process} if $\langle W, h\rangle$ is a standard Wiener process for each $h\in \mathbb R^d$ such that $\|h\|=1$.
The following lemma is a multidimensional variation of \cite[(3.2.19)]{KS}.
\begin{lemma}\label{lemma:covstochintwrtcylbrmo}
 Let $X = \mathbb R$, $d\geq 1$, $W$ be a standard $d$-dimensional Wiener process, $\Phi, \Psi:\mathbb R_+ \times \Omega \to \mathcal L(\mathbb R^d, \mathbb R)$ be elementary progressive. Then for all $t\geq 0$ a.s.
 \[
  [\Phi \cdot W, \Psi \cdot W]_t = \int_0^t \langle \Phi^*(s), \Psi^*(s) \rangle\ud s.
 \]
\end{lemma}
The reader can find more on stochastic integration with respect to a Wiener process in the Hilbert space case in  \cite{DPZ}, in the case of Banach spaces with a martingale type~2 in \cite{Brz1}, and in the UMD case in \cite{NVW}. Notice that the last mentioned work provides sharp $L^p$-estimates for stochastic integrals for the broadest till now known class of spaces. 

\subsection{Brownian representation}  The following theorem can be found in \cite[Theorem 3.4.2]{KS} (see also \cite{SV,Yar16Br}).

\begin{theorem}\label{thm:Brrepres}
 Let $d\geq 1$, $M:\mathbb R_+ \times \Omega \to \mathbb R^d$ be a continuous martingale such that $[M]$ is a.s.\ absolutely continuous with respect to the Lebesgue measure on $\mathbb R_+$. Then there exist an enlarged probability space $(\widetilde{\Omega}, \widetilde {\mathcal F}, \widetilde{\mathbb P})$ with an enlarged filtration $\widetilde{\mathbb F} = (\widetilde F_t)_{t\geq 0}$, a $d$-dimensional standard Wiener process $W:\mathbb R_+ \times \widetilde{\Omega}\to \mathbb R^d$ which is defined on the filtration $\widetilde {\mathbb F}$, and an $\widetilde {\mathbb F}$-progressively measurable $\Phi :\mathbb R_+\times \widetilde{\Omega} \to \mathcal L(\mathbb R^d)$ such that $M = \Phi \cdot W$.
\end{theorem}

\subsection{Lebesgue measure}
Let $X$ be a finite dimensional Banach space. Then according to Theorem 2.20 and Proposition 2.21 in \cite{FolHarm} there exists a unique  translation-invariant measure $\lambda_X$ on $X$ such that $\lambda_X(\mathbb B_X) = 1$ for the unit ball $\mathbb B_X$ of $X$. We will call $\lambda_X$ the {\em Lebesgue measure}.

\section{UMD Banach spaces and martingale decompositions}

Let $X$ be a Banach space, $1<p<\infty$. In this section we will show that the Meyer-Yoeurp and Yoeurp decompositions for $X$-valued $L^p$-martingales take place if and only if $X$~has the UMD property.

\subsection{Meyer-Yoeurp decomposition in UMD case}\label{subsec:MeyYoudec}
This subsection is devoted to the ge\-ne\-ra\-li\-za\-tion of Meyer-Yoeurp decomposition (see Remark \ref{rem:YoeMeyR}) to the UMD Banach space case:

\begin{theorem}[Meyer-Yoeurp decomposition]\label{thm:Meyer-Yoeurp}
 Let $X$ be a UMD Banach space, $p\in (1,\infty)$, $M:\mathbb R_+ \times \Omega \to X$ be an $L^p$-martingale. Then there exist unique martingales $M^d, M^c:\mathbb R_+ \times \Omega \to X$ such that $M^d$ is purely discontinuous, $M^c$ is continuous, $M^c_0 = 0$ and $M=M^d + M^c$. Moreover, then for all $t\geq 0$
 \begin{equation}\label{eq:Meyer-Yoeurp}
 (\mathbb E\|M^d_t\|^p)^{\frac 1p} \leq \beta_{p,X}(\mathbb E \|M_t\|^p)^{\frac 1p},\;\;\;\;\;
 (\mathbb E\|M^c_t\|^p)^{\frac 1p} \leq \beta_{p,X}(\mathbb E \|M_t\|^p)^{\frac 1p}.
 \end{equation}
\end{theorem}

The proof of the theorem consists of several steps. First we introduce the main tool of our proof -- the Burkholder function.

\begin{definition}
 Let $E$ be a linear space with a scalar field $\mathbb R$.
 \begin{itemize}
  \item[(i)] A function $f:E\times E \to \mathbb R$ is called {\em biconcave} if for each $x, y \in E$ one has that the mappings $e\mapsto f(x, e)$ and $e\mapsto f(e,y)$ are concave.
  \item[(ii)] A function $f:E\times E \to \mathbb R$ is called {\em zigzag-concave} if for each $x, y\in E$ and $\eps \in \mathbb R$ such that $|\eps| \leq 1$, the function $z\mapsto f(x+z, y+\eps z)$ is concave.
 \end{itemize}
\end{definition}

The following theorem is a small variation of \cite{Burk86} and \cite[Theorem 4.5.6]{HNVW1}, and has been proven in \cite{Y17FourUMD}.

\begin{theorem}[Burkholder]\label{thm:Burkholder}
 For a Banach space $X$ the following are equivalent
 \begin{enumerate}
  \item $X$ is a UMD Banach space;
  \item for each $p\in (1,\infty)$ there exists a constant $\beta$ and a zigzag-concave function $U:X\times X \to \mathbb R$ such that
  \begin{equation}\label{eq:ineqonU}
     U(x,y)\geq \|y\|^p - \beta^p\|x\|^p,\;\;\;x,y\in X.
  \end{equation}
 \end{enumerate}
 The smallest admissible $\beta$ for which such $U$ exists is $\beta_{p, X}$.
\end{theorem}

\begin{remark}\label{rem:propertiesofU}
 Fix a UMD space $X$ and $p\in (1,\infty)$. A special zigzag-concave function $U$ from Theorem \ref{thm:Burkholder} have been obtained in \cite[Theorem 4.5.6]{HNVW1}. We will call this function {\em the Burkholder function}. For the convenience of the reader we leave out the construction of the Burkholder function. The following properties of the Burkholder function $U$ were demonstrated in \cite[Section 3]{Y17FourUMD}:
 \begin{itemize}
  \item [\rm{(A)}] $U(\alpha x, \alpha y) = |\alpha|^p U(x, y)$ for all $x,y\in X$, $\mathbb \alpha \in \mathbb R$.
  \item [\rm{(B)}] $U(x,\alpha x)\leq 0$ for all $x\in X$, $\alpha \in [-1,1]$.
  \item [\rm{(C)}] $U$ is continuous.
 \end{itemize}
\end{remark}

\begin{remark}\label{rem:propertiesofV}
 Fix a UMD space $X$ and $p\in (1,\infty)$. Let the Burkholder function $U$ be as in Remark \ref{rem:propertiesofU}. Then there exists a biconcave function $V:X\times X \to \mathbb R$ such that 
 \begin{equation}\label{eq:VthroughU}
 V(x,y) = U\Bigl(\frac{x-y}{2},\frac{x+y}2\Bigr),\;\;\; x,y\in X.
\end{equation}
In \cite[Section 3]{Y17FourUMD} the following properties of $V$ have been explored:
 \begin{itemize}
  \item [\rm{(A)}] For each $x,y\in X$ and $a,b \in \mathbb R$ such that $|a+b|\leq |a-b|$ one has that the function
$$
z\mapsto V(x+az,y+bz) = U\Bigl(\frac{x-y}{2} + \frac{(a-b)z}2,\frac{x+y}2 + \frac{(a+b)z}2\Bigr)
$$
is concave.
  \item [\rm{(B)}] $V$ is continuous.
  \item [\rm{(C)}] Let $X$ be finite dimensional. Then $x\mapsto V(x, y)$ and $y\mapsto V(x,y)$ are a.s.\ Fr\'echet-differentiable with respect to the Lebesgue measure $\lambda_X$, and for a.a.\ $(x,y)\in X\times X$ for each $u,v\in X$ there exists the directional derivative $\frac{\partial V(x+tu,y+tv)}{\partial t}$. Moreover, 
\begin{equation}\label{eq:dirderivative}
 \frac{\partial V(x+tu,y+tv)}{\partial t} = \langle \partial_x V(x,y), u\rangle+\langle \partial_y V(x,y), v\rangle,
\end{equation}
where $\partial_x V$ and $\partial_y V$ are the corresponding Fr\'echet derivatives with respect to the first and the second variable. 
\item [\rm{(D)}] Let $X$ be finite dimensional. Then for a.e.\ $(x,y)\in X\times X$, for all $z\in X$ and real-valued $a$ and $b$ such that $|a+b|\leq |a-b|$
\begin{equation}\label{eq:dirderivativeV}
\begin{split}
  V(x+az,y+bz)&\leq V(x,y) + \frac{\partial V(x+atz,y+btz)}{\partial t}\\
 &= V(x,y)+a\langle \partial_x V(x,y), z\rangle+b\langle \partial_y V(x,y), z\rangle.
 \end{split}
\end{equation}
\item [\rm{(E)}] Let $X$ be finite dimensional. Then there exists $C>0$ which depends only on $V$ such that for a.e.\ pair $x, y\in X$, $\|\partial_xV(x, y)\|, \|\partial_yV(x, y)\| \leq C(\|x\|^{p-1} + \|y\|^{p-1})$.
 \end{itemize}
\end{remark}

\begin{definition}\label{def:corrbasis}
 Let $d$ be a natural number, $E$ be a $d$-dimensional linear space, $(e_n)_{n=1}^d$ be a basis of~$E$. Then $(e_n^*)_{n=1}^d\subset E^*$ is called the {\em corresponding dual basis} of $(e_n)_{n=1}^d$ if $\langle e_n, e_m^*\rangle = \delta_{nm}$ for each $m,n=1,\ldots, d$.
\end{definition}

 Note that the corresponding dual basis is uniquely determined. Moreover, if $(e_n^*)_{n=1}^d$ is the corresponding dual basis of $(e_n)_{n=1}^d$, then, the other way around, $(e_n)_{n=1}^d$ is the corresponding dual basis of $(e_n^*)_{n=1}^d$ (here we identify $E^{**}$ with $E$ in the natural way).
 
 \begin{lemma*}\label{lemma:traceindep}
 Let $d$ be a natural number, $E$ be a $d$-dimensional linear space. Let $V: E\times E \to \mathbb R$ and $W:E^* \times E^* \to \mathbb R$ be two bilinear functions. Then the expression
 \begin{equation}\label{eq:traceindep}
  \sum_{n, m=1}^d V(e_n, e_m) W(e_n^*, e_m^*)
 \end{equation}
does not depend on the choice of basis $(e_n)_{n=1}^d$ of $E$ (here $(e_n^*)_{n=1}^d$ is the corresponding dual basis of $(e_n)_{n=1}^d$).
\end{lemma*}

The following It\^o formula is a version of \cite[Theorem 26.7]{Kal} that does not use the Euclidean structure of a finite dimensional Banach space. The proof can be found in \cite{Y17FourUMD}.

\begin{theorem}[It\^o formula]\label{thm:itoformula}
 Let $d$ be a natural number, $X$ be a $d$-dimensional Banach space, $f\in C^2(X)$, $M:\mathbb R_+ \times \Omega \to X$ be a martingale. Let $(x_n)_{n=1}^d$ be a basis of $X$, $(x_n^*)_{n=1}^d$ be the corresponding dual basis. Then for each $t\geq 0$
 \begin{equation}\label{eq:itoformula}
  \begin{split}
   f(M_t) = f(M_0)&+ \int_0^t \langle \partial_xf(M_{s-}), \ud M_s\rangle\\
&+ \frac 12 \int_0^t \sum_{n,m=1}^d f_{x_n, x_m}(M_{s-})\ud[\langle M, x_n^*\rangle,\langle M, x_m^*\rangle]_s^c\\
&+ \sum_{s\leq t}(\Delta f(M_s) - \langle  \partial_xf(M_{s-}), \Delta M_{s}\rangle).
  \end{split}
 \end{equation}
\end{theorem}

\begin{proposition}\label{prop:approxXbyYvalued}
 Let $X$ be a finite dimensional Banach space, $p\in (1,\infty)$. Let $Y = X \oplus \mathbb R$ be a Banach space such that $\|(x, r)\|_Y = (\|x\|^p_X + |r|^p)^{\frac 1p}$. Then $\beta_{p, Y} = \beta_{p, X}$. Moreover, if $M:\mathbb R_+ \times \Omega \to X$ is a martingale on a probability space $(\Omega, \mathcal F, \mathbb P)$ with a filtration $\mathbb F = (\mathcal F_t)_{t\geq 0}$, then there exists a sequence $(M^m)_{m\geq 1}$ of $Y$-valued martingales on an enlarged probability space $(\overline{\Omega}, \overline{\mathcal F}, \overline{\mathbb P})$ with an enlarged filtration $\overline{\mathbb F} = (\overline{\mathcal F}_t)_{t\geq 0}$ such that
 \begin{enumerate}
  \item $M^m_t$ has absolutely continuous distributions with respect to the Lebesgue measure on $Y$ for each $m\geq 1$ and $t\geq 0$;
  \item $M^m_t \to (M_t,0)$ pointwise as $m\to \infty$ for each $t\geq 0$;
  \item if for some $t\geq 0$ $\mathbb E\|M_t\|^p<\infty$, then for each $m\geq 1$ one has that $\mathbb E\|M^m_t\|^p<\infty$ and $\mathbb E\|M^m_t-(M_t,0)\|^p\to 0$ as $m\to \infty$;
  \item if $M$ is continuous, then $(M^m)_{m\geq 1}$  are continuous as well,
  \item if $M$ is purely discontinuous, then $(M^m)_{m\geq 1}$ are purely discontinuous as well.
 \end{enumerate}
\end{proposition}

\begin{proof}
 The proof of (1)-(3) follows from \cite{Y17FourUMD}, while (4) and (5) follow from the construction of $M^m$ and $N^m$ given in \cite{Y17FourUMD}.
\end{proof}

\begin{remark}\label{rem:sumofapprox}
 Notice that the construction in \cite{Y17FourUMD} also allows us to sum these approximations for different martingales. Namely, if $M$ and $N$ are two $X$-valued martingales, then we can construct the corresponding $Y$-valued martingales $(M^m)_{m\geq 1}$ and $(N^m)_{m\geq 1}$ as in Proposition \ref{prop:approxXbyYvalued} in such a way that $M^m_t + N^m_t$ has an absolutely continuous distribution for each $t\geq 0$ and $m\geq 1$.
\end{remark}

\begin{proof}[Proof of Theorem \ref{thm:Meyer-Yoeurp}]
{\em Step 1: finite dimensional case.} Let $X$ be finite dimensional. Then $M^d$ and $M^c$ exist due to Remark \ref{rem:YoeMey}. Without loss of generality $\mathcal F_t = \mathcal F_{\infty}$, $M^d_t = M^d_{\infty}$ and $M^c_t = M^c_{\infty}$. Let $d$ be the dimension of $X$.
 
 Let $\vertiii{\cdot}$ be a Euclidean norm on $X$. Then $(X, \vertiii{\cdot})$ is a Hilbert space, and by Remark \ref{rem:qvcont} the quadratic variation $[M^c]$ exists and has a continuous version. Let us show that without loss of generality we can suppose that $[M^c]$ is a.s.\ absolutely continuous with respect to the Lebesgue measure on $\mathbb R_+$. Let $A:\mathbb R_+ \times \Omega \to \mathbb R_+$ be as follows: $A_t = [M^c]_t + t$. Then $A$ is strictly increasing continuous, $A_0=0$ and $A_{\infty} = \infty$ a.s. Let the time-change $\tau = (\tau_s)_{s\geq 1}$ be defined as in Theorem \ref{thm:apptimechange}. Then by Theorem \ref{thm:apptimechange}, $M^c \circ \tau$ is a continuous martingale, $M^d\circ \tau$ is a purely discontinuous martingale, $(M^c\circ \tau)_0=0$, $(M^d\circ \tau)_0=M^d_0$ and due to the Kazamaki theorem \cite[Theorem 17.24]{Kal}, $[M^c\circ \tau] = [M^c]\circ \tau$. Therefore for all $t>s\geq 0$ by Theorem \ref{thm:apptimechange} and the fact that $\tau_t \geq \tau_s$ a.s.
 \begin{align*}
  [M^c\circ \tau]_t - [M^c\circ \tau]_s = [M^c]_{\tau_t} - [M^c]_{\tau_s} &\leq [M^c]_{\tau_t} - [M^c]_{\tau_s} + (\tau_t - \tau_s)\\
  & = ([M^c]_{\tau_t}+\tau_t) - ([M^c]_{\tau_s} + \tau_s)\\
  &= A_{\tau_t} - A_{\tau_s}=t-s.
 \end{align*}
 Hence $[M^c\circ \tau]$ is a.s.\ absolutely continuous with respect to the Lebesgue measure on~$\mathbb R_+$. Moreover, $(M^i\circ \tau)_{\infty} = M^i_{\infty}$, $i\in\{c,d\}$, so this time-change argument does not affect~\eqref{eq:Meyer-Yoeurp}. Hence we can redefine $M^c := M^c\circ \tau$, $M^d := M^d \circ \tau$, $\mathbb F = (\mathcal F_s)_{s\geq 0} := \mathbb G = (\mathcal F_{\tau_s})_{s\geq 0}$.
 
 Since $[M^c]$ is a.s.\ absolutely continuous with respect to the Lebesgue measure on~$\mathbb R_+$ and thanks to Theorem \ref{thm:Brrepres}, we can extend $\Omega$ and find a $d$-dimensional Wiener process $W:\mathbb R_+\times \Omega \to \mathbb R^d$ and a stochastically integrable progressively measurable function $\Phi: \mathbb R_+ \times\Omega \to \mathcal L(\mathbb R^d, X)$ such that $M^c = \Phi \cdot W$. 
 
 Let $U:X \times X \to \mathbb R$ be the Burkholder function that was discussed in Remark~\ref{rem:propertiesofU} and Remark \ref{rem:propertiesofV}. Let us show that $\mathbb E U(M_t, M^d_t)\leq 0$. 
 
 Due to Proposition \ref{prop:approxXbyYvalued} and Remark \ref{rem:sumofapprox} we can assume that $M^c_s$, $M^d_s$ and $M_s = M^d_s + M^c_s$ have absolutely continuous distributions with respect to the Lebesgue measure $\lambda_X$ on $X$ for each $s\geq 0$. Let $(x_n)_{n=1}^d$ be a basis of $X$, $(x_n^*)_{n=1}^d$ be the corresponding dual basis of $X^*$ (see Definition \ref{def:corrbasis}). By the It\^o formula \eqref{eq:itoformula},
 \begin{equation}\label{eq:EU(M^d+M^c,M^d)}
  \begin{split}
     \mathbb E U(M_t,M^d_t) = \mathbb E U(M_0,M^d_0) &+ \mathbb E \int_{0}^t \langle \partial_x U(M_{s-},M^d_{s-}),\ud M_s\rangle\\
  &+\mathbb E \int_{0}^t \langle \partial_y U(M_{s-},M^d_{s-}),\ud M^d_s\rangle + \mathbb E I_1 + \mathbb E I_2,
  \end{split}
 \end{equation}
where
\begin{align*}
 I_1 &= \sum_{0<s\leq t}[\Delta U(M_s,M^d_s)-\langle \partial_x U(M_{s-},M^d_{s-}),\Delta M_s\rangle - \langle \partial_y U(M_{s-},M^d_{s-}),\Delta M^d_s\rangle],\\
 I_2 &= \frac 12 \int_0^t\sum_{i,j=1}^d U_{x_i, x_j}(M_{s-}, M^d_{s-})\ud[\langle M, x_i^*\rangle, \langle M, x_j^*\rangle]_s^c\\
 &\quad =\frac 12 \int_0^t\sum_{i,j=1}^d U_{x_i, x_j}(M_{s-}, M^d_{s-})\langle \Phi^*(s) x_i^*, \Phi^*(s) x_j^*\rangle\ud s.
\end{align*}
(Recall that by \eqref{eq:VthroughU} and Remark \ref{rem:propertiesofV}(C), $U$ is Fr\'echet-differentiable a.s.\ on $X\times X$, hence $\partial_x U$ and $\partial_y U$ are well-defined. Moreover, $U$ is zigzag-concave, so $U$ is concave in the first variable, and therefore the second-order derivatives $U_{x_i, x_j}$ in the first variable are well-defined and exist a.s.\ on $X \times X$ by the Alexandrov theorem \cite[Theorem 6.4.1]{EG}.) The last equality holds due to Theorem \ref{thm:itoformula} and the fact that by Lemma \ref{lemma:covstochintwrtcylbrmo} for all $s\geq 0$~a.s.\
\begin{align*}
 [\langle M, x_i^*\rangle, \langle M, x_j^*\rangle]_s^c = [\langle \Phi\cdot W, x_i^*\rangle, \langle \Phi\cdot W, x_j^*\rangle]_s &= [(\Phi^*x_i^*)\cdot W, (\Phi^*x_j^*)\cdot W]_s\\
 &=\int_0^s \langle\Phi^*(r)x_i^*, \Phi^*(r)x_j^*\rangle\ud r.
\end{align*}
Let us first show that $I_1\leq 0$ a.s. Indeed, since $M^d$ is a purely discontinuous part of $M$, then by Definition \ref{def:purelydiscXvalued} $\langle M^d,x^*\rangle$ is a purely discontinuous part of $\langle M,x^*\rangle$, and due to Remark \ref{rem:YoeMeyR} a.s.\ for each $t\geq 0$
$$
\Delta |\langle M^d,x^*\rangle|_t^2= \Delta [\langle M^d,x^*\rangle]_t= \Delta [\langle M,x^*\rangle]_t=\Delta |\langle M,x^*\rangle|_t^2
$$ 
for each $x^* \in X^*$. Thus for each $s\geq 0$ by \eqref{eq:dirderivative} and \eqref{eq:dirderivativeV} $\mathbb P$-a.s.
\begin{align*}
 &\Delta U(M_s,M^d_s)-\langle \partial_x U(M_{s-},M^d_{s-}),\Delta M_s\rangle - \langle \partial_y U(M_{s-},M^d_{s-}),\Delta M^d_s\rangle\\
 &= V(M_{s-}+M^d_{s-}+2\Delta M_s,M^d_{s-}-M_{s-})- V(M_{s-}+M^d_{s-},M^d_{s-}-M_{s-})\\ 
 &\quad-\langle \partial_x V(M_{s-}+M^d_{s-},M^d_{s-}-M_{s-}),2\Delta M_s\rangle \leq 0,
\end{align*}
so $I_1\leq 0$ a.s., and $\mathbb E I_1 \leq 0$.
Now we show that
\[
 \mathbb E \Bigl(\int_{0}^t \langle \partial_x U(M_{s-},M^d_{s-}),\ud M_s\rangle + \int_{0}^t \langle \partial_y U(M_{s-},M^d_{s-}),\ud M^d_s\rangle\Bigr) = 0.
\]
Indeed,
\begin{align*}
 \int_{0}^t \langle \partial_x U(M_{s-},M^d_{s-}),\ud M_s\rangle &+ \int_{0}^t \langle \partial_y U(M_{s-},M^d_{s-}),\ud M^d_s\rangle\\
 &=\int_{0}^t \langle \partial_x V(M_{s-}+M^d_{s-},M^d_{s-}-M_{s-}),\ud (M_s+M^d_s)\rangle \\
 &+ \int_{0}^t \langle  \partial_y V(M_{s-}+M^d_{s-},M^d_{s-}-M_{s-}),\ud (M^d_s-M_s)\rangle
\end{align*}
so by Lemma \ref{lemma:stochintmoment} and Remark \ref{rem:propertiesofV}(E) it is a martingale which starts at zero, hence its expectation is zero.

Finally let us show that $I_2 \leq 0$ a.s. Fix $s\in [0,t]$ and $\omega \in \Omega$. Then $x^* \mapsto \|\Phi^*(s, \omega) x^*\|^2$ defines a nonnegative definite quadratic form on $X^*$, and since any nonnegative quadratic form defines a Euclidean seminorm, there exists a basis $(\tilde x_n^*)_{n=1}^d$ of $X^*$ and a $\{0, 1\}$-valued sequence $(a_n)_{n=1}^d$ such that
\[
 \langle \Phi^*(s, \omega) \tilde x_n^*, \Phi^*(s, \omega) \tilde x_m^* \rangle = a_n\delta_{mn},\;\;\; m,n= 1,\ldots,d.
\]
Let $(\tilde x_n)_{n=1}^d$ be the corresponding dual basis of $X$ as it is defined in Definition \ref{def:corrbasis}. Then due to Lemma~\ref{lemma:traceindep} and the linearity of $\Phi$ and directional derivatives of $U$ (we skip $s$ and $\omega$ for the simplicity of the expressions)
\begin{equation*}
 \begin{split}
   \sum_{i,j=1}^d U_{x_i, x_j}(M_{s-}, M^d_{s-})\langle \Phi^* x_i^*, \Phi^* x_j^*\rangle &=\sum_{i,j=1}^d U_{\tilde x_i,\tilde x_j}(M_{s-}, M^d_{s-})\langle \Phi^*\tilde x_i^*, \Phi^* \tilde x_j^*\rangle\\
   &=\sum_{i=1}^d U_{\tilde x_i,\tilde x_i}(M_{s-}, M^d_{s-})\|\Phi^*\tilde x_i^*\|^2.
 \end{split}
\end{equation*}
Recall that $U$ is zigzag-concave, so $t\mapsto U(x+t\tilde x_i, y)$ is concave for each $x,y\in X$, $i=1,\ldots,d$. Therefore $U_{\tilde x_i,\tilde x_i}(M_{s-}, M^d_{s-}) \leq 0$ a.s., and a.s.
\[
 \sum_{i=1}^d U_{\tilde x_i,\tilde x_i}(M_{s-}(\omega), M^d_{s-}(\omega))\|\Phi^*(s,\omega)\tilde x_i^*\|^2\leq 0.
\]
Consequently, $I_2\leq 0$ a.s., and by \eqref{eq:EU(M^d+M^c,M^d)}, Remark \ref{rem:propertiesofU}(B) and the fact that $M^d_0=M_0$
$$
\mathbb E U(M_t, M^d_t)\leq \mathbb E U(M_0, M_0)\leq 0.
$$ 
By \eqref{eq:ineqonU},
 $\mathbb E \|M^d_t\|^p - \beta_{p, X}^p\mathbb E\|M_t\|^p \leq \mathbb E U(M_t, M^d_t) \leq 0$,
so the first part of \eqref{eq:Meyer-Yoeurp} holds. 

The second part of \eqref{eq:Meyer-Yoeurp} follows from the same machinery applied for $V$. Namely, one can analogously show that
\[
 \mathbb E\|M^c_t\|^p - \beta_{p, X}^p \mathbb E \|M_t\|^p \leq \mathbb E U(M_t, M^c_t) = \mathbb E V(M^d + 2M^c, -M^d) \leq 0
\]
by using a $V$-version of \eqref{eq:EU(M^d+M^c,M^d)}, inequality \eqref{eq:dirderivativeV}, and the fact that $V$ is concave in the first variable a.s.\ on $X\times X$.

{\em Step 2: general case.} Without loss of generality we set $\mathcal F_{\infty} = \mathcal F_t$. Let $M_t = \xi$. If $\xi$ is a simple function, then it takes its values in a finite dimensional subspace $X_0$ of $X$, and therefore $(M_s)_{s\geq 0} = (\mathbb E(\xi|\mathcal F_s))_{s\geq 0}$ takes its values in $X_0$ as well, so the theorem and \eqref{eq:Meyer-Yoeurp} follow from Step 1.
 
 Now let $\xi$ be general. Let $(\xi_n)_{n\geq 1}$ be a sequence of simple $\mathcal F_t$-measurable functions in $L^p(\Omega; X)$ such that $\xi_n \to \xi$ as $n\to \infty$ in $L^p(\Omega; X)$. For each $n\geq 1$ define $\mathcal F_t$-measurable $\xi_n^d$ and $\xi_n^c$ such that
 \begin{equation}\label{eq:M^dnM^cn}
  \begin{split}
   M^{d, n} &= (M^{d, n}_s)_{s\geq 0}=(\mathbb E(\xi_n^d|\mathcal F_{s}))_{s\geq 0},\\
   M^{c,n} &= (M^{c, n}_s)_{s\geq 0}=(\mathbb E(\xi_n^c|\mathcal F_{s}))_{s\geq 0}
  \end{split}
 \end{equation}
are the respectively purely discontinuous and continuous parts of martingale $M^n=(\mathbb E(\xi_n|\mathcal F_{s}))_{s\geq 0}$ as in Remark \ref{rem:YoeMey}. Then due to Step 1 and \eqref{eq:Meyer-Yoeurp}, $(\xi_n^d)_{n\geq 1}$ and $(\xi_n^c)_{n\geq 1}$ are Cauchy sequences in $L^p(\Omega; X)$. Let $\xi^c := L^p-\lim_{n\to \infty} \xi^c_n$ and $\xi^d := L^p-\lim_{n\to\infty} \xi^d_n$. Define the $X$-valued $L^p$-martingales $M^d$ and $M^c$ by
 \begin{align*}
  M^d = (M^d_s)_{s\geq 0} := (\mathbb E(\xi^d|\mathcal F_{s}))_{s\geq 0},\;\;\;
    M^c = (M^c_s)_{s\geq 0} := (\mathbb E(\xi^c|\mathcal F_{s}))_{s\geq 0}.
 \end{align*}
Thanks to Proposition \ref{thm:M^p,discomplete}, $M^d$ is purely discontinuous, and due to Proposition~\ref{prop:M^p,ciscomplete} $M^c$ is continuous and $M^c_0=0$, so $M = M^d + M^c$ is the desired decomposition.

 The uniqueness of the decomposition follows from Lemma \ref{lemma:contpuredisczero}. For estimates \eqref{eq:Meyer-Yoeurp} we note that by Step 1, \eqref{eq:Meyer-Yoeurp} applied for Step 1, and \cite[Proposition 4.2.17]{HNVW1} for each $n\geq 1$
 $$
 (\mathbb E\|\xi_n^d\|^p)^{\frac 1p}\leq \beta_{p,X}(\mathbb E\|\xi_n\|^p)^{\frac 1p},\;\;\;\;
 (\mathbb E\|\xi_n^c\|^p)^{\frac 1p}\leq \beta_{p,X}(\mathbb E\|\xi_n\|^p)^{\frac 1p},
 $$ 
 and it remains to let $n\to \infty$.
\end{proof}

\begin{remark}\label{cor:approxofconandpurediscmartbyfd}
 Let $X$ be a UMD Banach space, $1<p<\infty$, $M:\mathbb R_+ \times \Omega \to X$ be continuous (resp.\ purely discontinuous) $L^p$-martingale. Then there exists a sequence $(M^n)_{n\geq 1}$ of continuous (resp.\ purely discontinuous) $X$-valued $L^p$-martingales such that $M^n$ takes its values is a finite dimensional subspace of $X$ for each $n \geq 1$ and $M^n_{\infty} \to M_{\infty}$ in $L^p(\Omega; X)$ as $n\to \infty$. Such a sequence can be provided e.g.\ by \eqref{eq:M^dnM^cn}.
\end{remark}

We have proven the Meyer-Yoeurp decomposition in the UMD setting. Next we prove a converse result which shows the necessity of the UMD property.

\begin{theorem}\label{thm:exampleforlowboundofA}
 Let $X$ be a finite dimensional Banach space, $p\in(1,\infty)$, $\delta\in(0, (\beta_{p, X}-1)\wedge 1)$. Then there exist a purely discontinuous martingale $M^d:\mathbb R_+ \times \Omega \to X$, a continuous martingale $M^c:\mathbb R_+ \times \Omega \to X$ such that $\mathbb E \|M^d_{\infty}\|^p, \mathbb E \|M^c_{\infty}\|^p<\infty$, $M^d_0 = M^c_0=0$, and for $M = M^d + M^c$ and $i\in\{c,d\}$ the following hold
 \begin{equation}\label{eq:exampleforlowboundofA}
 (\mathbb E \|M^i_{\infty}\|^p)^{\frac 1p} \geq \Bigl(\frac{\beta_{p,X}-1}{2}-\delta\Bigr) (\mathbb E \|M_{\infty}\|^p)^{\frac 1p}.
 \end{equation}
\end{theorem}

Recall that by \cite[Proposition 4.2.17]{HNVW1} $\beta_{p, X} \geq \beta_{p, \mathbb R}= p^*-1 \geq 1$ for any UMD Banach space $X$ and $1<p<\infty$.

\begin{definition}
 A random variable $r:\Omega \to \{-1, 1\}$ is called a~{\em Rademacher variable} if $\mathbb P(r=1) = \mathbb P(r=-1) = \frac 12$.
\end{definition}

\begin{lemma*}\label{lem:contapproxradem}
 Let $\eps>0$, $p\in (1,\infty)$. Then there exists a continuous martingale $M:[0,1]\times \Omega \to [-1,1]$ with a symmetric distribution such that $\sign M_1$ is a~Rademacher random variable and
 \begin{equation}\label{eq:lemcontapproxradem}
  \|M_1-\sign M_1\|_{L^p(\Omega)}<\eps.
 \end{equation}
\end{lemma*}

We will need a definition of a Paley-Walsh martingale. 

\begin{definition}[Paley-Walsh martingales]\label{def:Paley-Walsh}
 Let $X$ be a Banach space. A discrete $X$-valued martingale $(f_n)_{n\geq 0}$ is called a {\em Paley-Walsh martingale} if there exist a sequence of independent Rademacher variables $(r_n)_{n\geq 1}$, a function $\phi_n:\{-1, 1\}^{n-1}\to X$ for each $n\geq 2$ and $\phi_1 \in X$ such that $df_n = r_n \phi_n(r_1,\ldots, r_{n-1})$ for each $n\geq 2$ and $df_1 = r_1 \phi_1$.
\end{definition}

\begin{remark}\label{rem:thesamelowerboundford_j}
 Let $X$ be a UMD space, $1<p<\infty$, $\delta>0$. Then using Proposition \ref{prop:UMD01sequence} one can construct a martingale difference sequence $(d_j)^n_{j=1}\in L^p(\Omega; X)$ and a $\{-1,1\}$-valued sequence
$(\varepsilon_j)^n_{j=1}$ such that
\begin{align*}
 \Bigl(\mathbb E \Bigl \| \sum^n_{j=1}\frac {\eps_j\pm1}{2}d_j\Bigr\|^p\Bigr )^{\frac 1p}& \geq \frac{\beta_{p,X}-\delta-1}{2} \Bigl(\mathbb E \Bigl \| \sum^n_{j=1}d_j\Bigr\|^p\Bigr )^{\frac 1p}.
\end{align*}
\end{remark}

\begin{proof}[Proof of Theorem \ref{thm:exampleforlowboundofA}]
   Denote $\frac{\beta_{p,X}-\delta-1}{2}$ by $\gamma_{p,X}^{\delta}$. By Proposition \ref{prop:UMD01sequence} there exists a natural number $N\geq 1$, a~discrete $X$-valued martingale $(f_n)_{n=0}^N$ such that $f_0 =0$, and a sequence of scalars $(\eps_n)_{n=1}^N$ such that $\eps_n\in \{0,1\}$ for each $n=1,\ldots, N$, such that
 \begin{equation}\label{eq:peoofofexampleforlowboundofA}
   \Bigl(\mathbb E \Bigl\| \sum^N_{n=1} \varepsilon_n df_n\Bigr\|^p\Bigr )^{\frac 1p}
\geq \gamma_{p,X}^{\delta} (\mathbb E \|f_N\|^p)^{\frac 1p}.
 \end{equation}
 According to \cite[Theorem 3.6.1]{HNVW1} we can assume that $(f_n)_{n=0}^N$ is a Paley-Walsh martingale. Let $(r_n)_{n=1}^N$ be a sequence of Rademacher variables and $(\phi_n)_{n=1}^N$ be a sequence of functions as in Definition \ref{def:Paley-Walsh}, i.e.\ be such that $f_n = \sum_{k=2}^n r_k\phi_k(r_1,\ldots,r_{k-1}) + r_1\phi_1$ for each $n=1, \ldots,N$. Without loss of generality we assume that
 \begin{equation}\label{eq:analogueofpeoofofexampleforlowboundofAprivetvsem}
  (\mathbb E\|f_{N}\|^p)^{\frac 1p} \geq 2.
 \end{equation}
For each $n=1, \ldots,N$ define a continuous martingale $M^n:[0,1]\times \Omega \to [-1,1]$ as in Lemma \ref{lem:contapproxradem}, i.e.\ a martingale $M^n$ with a symmetric distribution such that $\sign M^n_1$ is a Rademacher variable and
 \begin{equation}\label{eq:M^n_1-signM_n^1}
  \|M^n_1-\sign M^n_1\|_{L^p(\Omega)}<\frac{\delta}{KL},
 \end{equation}
 where $K =\beta_{p, X}N \max\{\|\phi_1\|, \|\phi_2\|_{\infty},\ldots, \|\phi_N\|_{\infty}\}$, and $L =2\beta_{p, X}$. Without loss of generality suppose that $(M^n)_{n=1}^N$ are independent. For each $n=1,\ldots, N$ set $\sigma_n = \sign M^n_1$. Define a martingale $M:[0, N+1]\times \Omega \to X$ in the following way:
 \[
M_t =
\begin{cases}
0, &\text{if}\;\; 0\leq t<1;\\
 M_{n-}+ M^n_{t-n}\phi_{n}(\sigma_1,\ldots,\sigma_{n-1}),&\text{if}\;\;  t\in[n,n+1)\;\; \text{and}\;\; \eps_n=0;\\
  M_{n-}+ \sigma_n \phi_{n}(\sigma_1,\ldots,\sigma_{n-1}),&\text{if}\;\; t\in[n,n+1)\;\; \text{and}\;\; \eps_n=1.
\end{cases}  
 \]
Let $M = M^d + M^c$ be the decomposition of Theorem \ref{thm:Meyer-Yoeurp}. Then 
\begin{align*}
 M^c_{N+1} &= \sum_{n=1}^NM^n_1\phi_{n}(\sigma_1,\ldots,\sigma_{n-1})\mathbf 1_{\eps_n=0},\\
 M^d_{N+1} &= \sum_{n=1}^N\sigma_n \phi_{n}(\sigma_1,\ldots,\sigma_{n-1})\mathbf 1_{\eps_n=1} =\sum_{n=1}^N\eps_n\sigma_n \phi_{n}(\sigma_1,\ldots,\sigma_{n-1}).
\end{align*}
 Notice that $(\sigma_n)_{n=1}^N$ is a sequence of independent Rademacher variables, so by~\eqref{eq:peoofofexampleforlowboundofA} and the discussion thereafter
\begin{equation}\label{eq:analogueofpeoofofexampleforlowboundofA}
  \Bigl(\mathbb E \Bigl\| \sum^N_{n=1} \varepsilon_n \,\sigma_n \phi_{n}(\sigma_1,\ldots,\sigma_{n-1})\Bigr\|^p\Bigr)^{\frac 1p}
\geq \gamma_{p,X}^{\delta}\Bigl(\mathbb E \Bigl\| \sum^N_{n=1} \sigma_n \phi_{n}(\sigma_1,\ldots,\sigma_{n-1})\Bigr\|^p\Bigr)^{\frac 1p}.
\end{equation}

Let us first show \eqref{eq:exampleforlowboundofA} with $i=d$. Note that by the triangle inequality, \eqref{eq:analogueofpeoofofexampleforlowboundofAprivetvsem}~and~\eqref{eq:M^n_1-signM_n^1}
\begin{equation}\label{eq:analogueofpeoofofexampleforlowboundofA12121}
 \begin{split}
  (\mathbb E\|M_{N+1}\|^p)^{\frac 1p} &\geq (\mathbb E\|f_{N}\|^p)^{\frac 1p} - \sum^N_{n=1} \Bigl(\mathbb E \Bigl\| (M^n_1-\sigma_n) \phi_{n}(\sigma_1,\ldots,\sigma_{n-1})\Bigr\|^p\Bigr)^{\frac 1p}\\
  &\geq 2-\frac{\delta}{KL}\cdot N \cdot\max\{\|\phi_1\|, \|\phi_2\|_{\infty},\ldots, \|\phi_N\|_{\infty}\} >1.
 \end{split}
\end{equation}
Therefore,
\begin{align*}
 (\mathbb E\|M^d_{N+1}\|^p)^{\frac 1p} 
  &= \Bigl(\mathbb E \Bigl\| \sum^N_{n=1} \eps_n \,\sigma_n \phi_{n}(\sigma_1,\ldots,\sigma_{n-1})\Bigr\|^p\Bigr)^{\frac 1p} \stackrel{(i)}\geq\gamma_{p,X}^{\delta} \Bigl(\mathbb E \Bigl\| \sum^N_{n=1} \sigma_n \phi_{n}(\sigma_1,\ldots,\sigma_{n-1})\Bigr\|^p\Bigr)^{\frac 1p} \\
  &\stackrel{(ii)}\geq\gamma_{p,X}^{\delta} \Bigl(\mathbb E \Bigl\| \sum^N_{n=1}\mathbf 1_{\eps_n = 1} \sigma_n  \phi_{n}(\sigma_1,\ldots,\sigma_{n-1}) + \sum^N_{n=1}\mathbf 1_{\eps_n = 0} M^n_1 \phi_{n}(\sigma_1,\ldots,\sigma_{n-1})
  \Bigr\|^p\Bigr)^{\frac 1p} \\
  &\quad -\gamma_{p,X}^{\delta}\sum^N_{n=1} \Bigl(\mathbb E \Bigl\| (M^n_1-\sigma_n) \phi_{n}(\sigma_1,\ldots,\sigma_{n-1})\Bigr\|^p\Bigr)^{\frac 1p}\\
  &\stackrel{(iii)}\geq\gamma_{p,X}^{\delta}(\mathbb E\|M_{N+1}\|^p)^{\frac 1p} - \frac{\delta}{L}\stackrel{(iv)}\geq\Bigl(\frac{\beta_{p,X}-1}{2}-\delta\Bigr)(\mathbb E\|M_{N+1}\|^p)^{\frac 1p} ,
\end{align*}
where $(i)$ follows from \eqref{eq:analogueofpeoofofexampleforlowboundofA}, $(ii)$ holds by the triangle inequality, $(iii)$ holds by \eqref{eq:M^n_1-signM_n^1}, and $(iv)$ follows from \eqref{eq:analogueofpeoofofexampleforlowboundofA12121}. By the same reason and Remark \ref{rem:thesamelowerboundford_j}, \eqref{eq:exampleforlowboundofA} holds for $i=c$.
\end{proof}

Let $p\in (1,\infty)$. Recall that $\mathcal M_X^p$ is a space of all $X$-valued $L^p$-martingales, $\mathcal M_X^{p, d}, \mathcal M_X^{p, c}\subset \mathcal M_X^{p}$ are its subspaces of purely discontinuous martingales and continuous martingales that start at zero respectively (see Subsection \ref{subsec:prelimmart}, \ref{subsec:prelimcontmart}, and \ref{subsec:prelimpdmart}).

\begin{theorem*}\label{thm:charofUMDbyMeyYoe}
 Let $X$ be a Banach space. Then $X$ is UMD if and only if for some (or, equivalently, for all) $p\in(1,\infty)$, for any probability space $(\Omega, \mathcal F,\mathbb P)$ with any filtration $\mathbb F = (\mathcal F_t)_{t\geq 0}$ that satisfies the usual conditions, $\mathcal M_X^{p} = \mathcal M_X^{p, d} \oplus \mathcal M_X^{p, c}$, and there exist projections $A^d, A^c\in \mathcal L(\mathcal M_X^{p})$ such that $\text{ran } A^d = \mathcal M_X^{p, d}$, $\text{ran } A^c = \mathcal M_X^{p, c}$, and for any $M \in \mathcal M_X^{p}$ the decomposition $M = A^d M + A^cM$
is the Meyer-Yoeurp decomposition from Theorem \ref{thm:Meyer-Yoeurp}. If this is the case, then
\begin{equation}\label{eq:charofUMDbyMeyYoe1}
  \|A^d\|\leq \beta_{p,X}\;\; \text{and}\;\;
  \|A^c\|\leq  \beta_{p,X}.
\end{equation}
Moreover, there exist  $(\Omega, \mathcal F,\mathbb P)$ and $\mathbb F = (\mathcal F_t)_{t\geq 0}$ such that
\begin{equation}\label{eq:charofUMDbyMeyYoe2}
  \|A^d\|,\|A^c\|\geq \frac {\beta_{p, X}-1}{2}\vee 1.
\end{equation}
\end{theorem*}

\begin{corollary}\label{cor:M_X^pd^*=M_X^*^p'dM_X^pc^*=M_X^*^p'c}
 Let $X$ be a UMD Banach space, $p\in (1,\infty)$. Let $i\in \{c,d\}$. Then $(\mathcal M_X^{p, i})^* \simeq \mathcal M_{X^*}^{p', i}$, and for each $M\in \mathcal M_{X^*}^{p', i}$ and $N\in \mathcal M_X^{p, i}$
 \begin{equation*}
   \langle M, N\rangle:= \mathbb E\langle M_{\infty}, N_{\infty}\rangle,\;\; \;\;
  \|M\|_{(\mathcal M_X^{p, i})^*}\eqsim_{p, X}\|M\|_{\mathcal M_{X^*}^{p', i}}.
 \end{equation*}
\end{corollary}

To prove the corollary above we will need the following lemma.

\begin{lemma}
 Let $X$ be a UMD Banach space, $p\in (1,\infty)$, $M \in \mathcal M_X^{p, d}$, $N \in \mathcal M_{X^*}^{p', c}$. Then $\mathbb E\langle M_{\infty}, N_{\infty}\rangle = 0$.
\end{lemma}

\begin{proof}
First suppose that $N_{\infty}$ takes it values in a finite dimensional subspace $Y$ of $X^*$. Let $d\geq 1$ be the dimension of $Y$, $(y_k)_{k=1}^d$ be the basis of $Y$. Then there exist $N^1,\ldots,N^d \in \mathcal M_{\mathbb R}^{p', c}$ such that $N = \sum_{k=1}^d N^k y_k$. Hence
\begin{equation}\label{eq:EM_inftyN_infty=0}
 \begin{split}
  E\langle M_{\infty}, N_{\infty}\rangle  =  E\Bigl\langle M_{\infty}, \sum_{k=1}^d N^k_{\infty} y_k\Bigr\rangle = \sum_{k=1}^d\mathbb E\langle M_{\infty}, y_k\rangle N^k_{\infty} \stackrel{(*)}= 0,
 \end{split}
\end{equation}
where $(*)$ holds due to Proposition \ref{thm:purdiscorthtoanycont1}.

Now turn to the general case. By Remark \ref{cor:approxofconandpurediscmartbyfd} for each $N \in \mathcal M_{X^*}^{p', c}$ there exists a se\-qu\-ence $(N^n)_{n\geq 1}$ of continuous martingales such that each of $N^n$ is in $\mathcal M_{X^*}^{p', c}$ and takes its valued in a finite dimensional subspace of $X^*$, and $N^n_{\infty} \to N_{\infty}$ in $L^{p'}(\Omega; X^*)$ as $n\to \infty$. Then due to \eqref{eq:EM_inftyN_infty=0},
  $ E\langle M_{\infty}, N_{\infty}\rangle = \lim_{n\to \infty}  E\langle M_{\infty}, N^n_{\infty}\rangle =0$,
so the lemma holds.
\end{proof}

\begin{proof}[Proof of Corollary \ref{cor:M_X^pd^*=M_X^*^p'dM_X^pc^*=M_X^*^p'c}]
 We will show only the case $i=d$, the case $i=c$ can be shown analogously.
 
 $\mathcal M_{X^*}^{p', d}\subset (\mathcal M_X^{p, d})^*$ and $\|M\|_{(\mathcal M_X^{p, d})^*} \leq \|M\|_{\mathcal M_{X^*}^{p', d}}$ for each $M \in \mathcal M_{X^*}^{p', d}$ thanks to the H\"older inequality. Now let us show the inverse. Let $f\in (\mathcal M_X^{p, d})^*$. Since due to Proposition \ref{thm:M^p,discomplete} $\mathcal M_X^{p, d}$ is a closed subspace of $\mathcal M_X^{p}$, by the Hahn-Banach theorem and Proposition \ref{prop:dualofMXp} there exists $L \in \mathcal M_{X^*}^{p'}$ such that
  $\mathbb E\langle L_{\infty}, N_{\infty}\rangle = f(N)$ for any $N \in \mathcal M_X^{p, d}$,
and $\|L\|_{\mathcal M_{X^*}^{p'}} = \|f\|_{ (\mathcal M_X^{p, d})^*}$. Let $L = L^d + L^c$ be the Meyer-Yoeurp decomposition of $L$ as in Theorem \ref{thm:Meyer-Yoeurp}. Then by \eqref{eq:Meyer-Yoeurp}
$$
\|L^d\|_{\mathcal M_{X^*}^{p', d}}\lesssim_{p, X}\|L\|_{\mathcal M_{X^*}^{p'}} = \|f\|_{ (\mathcal M_X^{p, d})^*}
$$ 
and
 $ \mathbb E\langle L^d_{\infty}, N_{\infty}\rangle =  \mathbb E\langle L_{\infty}, N_{\infty}\rangle$,
so the theorem holds.
\end{proof}

\subsection{Yoeurp decomposition of purely discontinuous martingales}\label{subsec:Youdec}
As Yoeurp shown in \cite{Yoe76}, one can provide further decomposition of a purely discontinuous martingale into two parts: a martingale with accessible jumps and a quasi-left continuous martingale. This subsection is devoted to the generalization of this result to a UMD case.

\begin{definition}
 Let $\tau$ be a stopping time. Then $\tau$ is called a {\em predictable stopping time} if there exists a sequence of stopping times $(\tau_n)_{n\geq 1}$ such that $\tau_n<\tau$ a.s.\ on $\{\tau>0\}$ for each $n\geq 1$ and $\tau_n \nearrow\tau$ a.s.
\end{definition}

\begin{definition}
 Let $\tau$ be a stopping time. Then $\tau$ is called a {\em totally inaccessible stopping time} if $\mathbb P\{\tau = \sigma < \infty\} =0$ for each predictable stopping time $\sigma$.
\end{definition}

\begin{definition}\label{def:accjumpsqlc}
 Let $A:\mathbb R_+ \times \Omega \to \mathbb R$ be an adapted c\`adl\`ag process. $A$ has {\em accessible jumps} if $\Delta A_{\tau}=0$ a.s.\ for any totally inaccessible stopping time $\tau$. $A$ is called {\em quasi-left continuous} if $\Delta A_{\tau}=0$ a.s.\ for any predictable stopping time $\tau$. 
\end{definition}

For the further information on the definitions given we refer the reader to \cite{Kal}.

\begin{remark}\label{rem:incprocdecom}
 According to \cite[Proposition 25.17]{Kal} one can show that for any pure jump increasing adapted c\`adl\`ag process $A:\mathbb R_+ \times \Omega \to \mathbb R$ there exist unique increasing adapted c\`adl\`ag processes $A^a, A^q:\mathbb R_+ \times \Omega \to \mathbb R$ such that $A^a$ has accessible jumps, $A^q$ is quasi-left continuous, $A^q_0 =0$ and $A=A^a + A^q$.
\end{remark}

The following decomposition theorem was shown by Yoeurp in \cite{Yoe76} (see also \cite[Corollary 26.16]{Kal}):
\begin{theorem}\label{thm:Yoeurpdec}
 Let $M:\mathbb R_+ \times \Omega \to \mathbb R$ be a purely discontinuous martingale. Then there exist unique purely discontinuous martingales $M^a,M^q:\mathbb R_+ \times \Omega \to \mathbb R$ such that $M^a$ is has accessible jumps, $M^q$ is quasi-left continuous, $M^q_0=0$ and $M = M^a + M^q$. Moreover, then $[M^a] = [M]^a$ and $[M^q] = [M]^q$.
\end{theorem}

\begin{corollary}\label{cor:M=M_0ifMbothaccjandqlc}
 Let $M:\mathbb R_+ \times \Omega \to \mathbb R$ be a purely discontinuous martingale which is both with accessible jumps and quasi-left continuous. Then $M = M_0$ a.s.
\end{corollary}

 \begin{proof}
 Without loss of generality we can set $ M_0=0$. Then $M = M+ 0 = 0+M$ are decompositions of $M$ into a sum of a martingale with accessible jumps and a quasi-left continuous martingale. Since by Theorem \ref{thm:Yoeurpdec} this decomposition is unique, $M=0$ a.s.
\end{proof}

\begin{proposition*}\label{prop:limthmforaccjqlc}
 Let $1<p<\infty$, $M:\mathbb R_+ \times \Omega \to \mathbb R$ be a purely discontinuous \mbox{$L^p$-mar}\-tin\-gale. Let $(M^n)_{n\geq 1}$ be a sequence of purely discontinuous martingales such that $M^n_{\infty} \to M_{\infty}$ in $L^p(\Omega)$. Then the following assertions hold
 \begin{itemize}
  \item [(a)] if $(M^n)_{n\geq 1}$ have accessible jumps, then $M$ has accessible jumps as well;
    \item [(b)] if $(M^n)_{n\geq 1}$ are quasi-left continuous martingales, then $M$ is quasi-left continuous as well.
 \end{itemize}
\end{proposition*}

\begin{definition}
 Let $X$ be a Banach space. A martingale $M:\mathbb R_+ \times \Omega \to X$ has {\em accessible jumps} if $\Delta M_{\tau}=0$ a.s.\ for any totally inaccessible stopping time $\tau$. A~martingale $M:\mathbb R_+ \times \Omega \to X$ is called {\em quasi-left continuous} if $\Delta M_{\tau}=0$ a.s.\ for any predictable stopping time $\tau$.
\end{definition}

\begin{lemma*}\label{lem:weakdefforaccjumpsandqlc}
 Let $X$ be a reflexive Banach space, $M:\mathbb R_+ \times \Omega \to X$ be a purely discontinuous martingale.
 \begin{itemize}
  \item [(i)] $M$ has accessible jumps if and only if for each $x^* \in X^*$ the martingale $\langle M, x^*\rangle$ has accessible jumps;
   \item [(ii)] $M$ is quasi-left continuous if and only if for each $x^* \in X^*$ the martingale $\langle M, x^*\rangle$ is quasi-left continuous.
 \end{itemize}
\end{lemma*}

\begin{definition}
 Let $X$ be a Banach space, $p\in (1,\infty)$. Then we define $\mathcal M_X^{p, q}\subset \mathcal M_X^{p,d}$ as a linear space of all $X$-valued purely discontinuous quasi-left continuous $L^p$-martingales which start at $0$. We define $\mathcal M_X^{p, a}\subset \mathcal M_X^{p,d}$ as a linear space of all $X$-valued purely discontinuous $L^p$-martingales with accessible jumps.
\end{definition}

\begin{proposition*}\label{prop:M^pqandM^paareclosed}
 Let $X$ be a Banach space, $1<p<\infty$. Then $\mathcal M_X^{p, q}$ and $\mathcal M_X^{p, a}$ are closed subspaces of $\mathcal M_X^{p, d}$.
\end{proposition*}

The following lemma follows from Corollary \ref{cor:M=M_0ifMbothaccjandqlc}.

\begin{lemma*}\label{lem:accjumps+qlc=cons}
 Let $X$ be a Banach space, $M:\mathbb R_+ \times\Omega \to X$ be a purely discontinuous martingale. Let $M$ be both with accessible jumps and quasi-left continuous. Then $M = M_0$ a.s. In other words, $\mathcal M_X^{p, q}\cap\mathcal M_X^{p, a} = 0$.
\end{lemma*}

The main theorem of this subsection is the following UMD variant of Theorem~\ref{thm:Yoeurpdec}.

\begin{theorem}\label{thm:YoeurpdecUMD}
 Let $X$ be a UMD Banach space, $M:\mathbb R_+ \times \Omega \to X$ be a purely discontinuous $L^p$-martingale. Then there exist unique purely discontinuous martingales $M^a,M^q:\mathbb R_+ \times \Omega \to X$ such that $M^a$ has accessible jumps, $M^q$ is quasi-left continuous, $M^q_0=0$ and $M = M^a + M^q$. Moreover, if this is the case, then for $i\in \{a,q\}$
 \begin{equation}\label{eq:thm:YoeurpdecUMDaq}
  (\mathbb E\|M^i_{\infty}\|^p)^{\frac 1p}\leq \beta_{p,X}(\mathbb E\|M_{\infty}\|^p)^{\frac 1p}.
 \end{equation}
\end{theorem}

\begin{proof}
{\em Step 1: finite dimensional case.} First assume that $X$ is finite dimensional. Then $M^a$ and $M^q$ exist and unique due to coordinate-wise applying of Theorem~\ref{thm:Yoeurpdec}.
Let $M = M^a + M^q$, $N= M^a$. Then for any $x^* \in X^*$, $t\geq 0$ by Theorem~\ref{thm:Yoeurpdec} and Lemma~\ref{lem:weakdefforaccjumpsandqlc}~a.s.
\begin{align*}
[\langle M, x^*\rangle]_t = [\langle M, x^*\rangle]^a_t + [\langle M, x^*\rangle]^q_t = [\langle M^a, x^*\rangle]_t + [\langle M^q, x^*\rangle]_t,
\end{align*}
and
\[
 [\langle N, x^*\rangle]_t = [\langle N, x^*\rangle]^a_t + [\langle N, x^*\rangle]^q_t = [\langle M^a, x^*\rangle]_t .
\]
Therefore a.s.
\[
 [\langle N, x^*\rangle]_t -[\langle N, x^*\rangle]_s \leq [\langle M, x^*\rangle]_t -[\langle M, x^*\rangle]_s,\;\;\; 0\leq s<t.
\]
Moreover $M_0 = N_0$.
Hence $N$ is weakly differentially subordinated to $M$ (see Section \ref{sec:weakdifsubandgenmart}), and \eqref{eq:thm:YoeurpdecUMDaq} for $i=a$ follows from \cite{Y17FourUMD}. By the same reason and since $M^q_0 = 0$, \eqref{eq:thm:YoeurpdecUMDaq} holds true for $i=q$.

{\em Step 2: general case.} 
 Now let $X$ be general. Let $\xi = M_{\infty}$.  Without loss of generality we set $\mathcal F_{\infty} = \mathcal F_t$. Let $(\xi_n)_{n\geq 1}$ be a sequence of simple $\mathcal F_t$-measurable functions in $L^p(\Omega; X)$ such that $\xi_n \to \xi$ as $n\to \infty$ in $L^p(\Omega; X)$. For each $n\geq 1$ define $\mathcal F_t$-measurable $\xi_n^d$ and $\xi_n^c$ such that $M^{d,n} = (\mathbb E(\xi_n^d|\mathcal F_{s}))_{s\geq 0}$ and $M^{c,n}=(\mathbb E(\xi_n^c|\mathcal F_{s}))_{s\geq 0}$ are respectively purely discontinuous and continuous parts of a martingale $(\mathbb E(\xi_n|\mathcal F_{s}))_{s\geq 0}$ as in Remark \ref{rem:YoeMey}. Then thanks to Theorem \ref{thm:Meyer-Yoeurp}, $\xi^d_n \to \xi$ and $\xi_n^c \to 0$ in $L^p(\Omega; X)$ as $n\to \infty$ since $M$ is purely discontinuous.
 
 Since for each $n\geq 1$ the random variable $\xi^d_n$ takes its values in a finite dimensional space, by Theorem \ref{thm:Yoeurpdec} there exist $\mathcal F_t$-measurable $\xi^a, \xi^q\in L^p(\Omega; X)$ such that purely discontinuous martingales  $M^{a,n}=(\mathbb E(\xi_n^a|\mathcal F_{s}))_{s\geq 0}$ and $M^{q,n}=(\mathbb E(\xi_n^q|\mathcal F_{s}))_{s\geq 0}$ are respectively with accessible jumps and quasi-left continuous, $\mathbb E(\xi_n^q|\mathcal F_{0})=0$, and the decomposition 
 $M^{d,n} = M^{a, n} + M^{q,n}$
is as in Theorem~\ref{thm:Yoeurpdec}. Since $(\xi^d_n)_{n\geq 1}$ is a Cauchy sequence in $L^p(\Omega; X)$, by Step 1 both $(\xi^a_n)_{n\geq 1}$ and $(\xi^q_n)_{n\geq 1}$ are Cauchy in $L^p(\Omega; X)$ as well. Let $\xi^a$ and $\xi^q$ be their limits. Define martingales $M^a, M^q:\mathbb R_+ \times \Omega \to X$ in the following way:
\begin{align*}
 M^a_s:= \mathbb E(\xi^a|\mathcal F_s),\;\; M^q_s:=\mathbb E(\xi^q|\mathcal F_s),\;\;\; s\geq 0.
\end{align*}
By Proposition \ref{prop:M^pqandM^paareclosed} $M^a$ is a martingale with accessible jumps, $M^q$ is quasi-left continuous, $M^q_0 = 0$ a.s., and therefore $M = M^a + M^q$ is the desired decomposition. Moreover, by Step 1 for each $n\geq 1$ and $i\in \{a,q\}$,
$(\mathbb E \|\xi^i_n\|^p)^{\frac 1p} \leq \beta_{p, X} (\mathbb E \|\xi^d_n\|^p)^{\frac 1p}$,
and hence the estimate \eqref{eq:thm:YoeurpdecUMDaq} follows by letting $n$ to infinity.

The uniqueness of the decomposition follows from Lemma \ref{lem:accjumps+qlc=cons}.
\end{proof}

The following theorem, as Theorem \ref{thm:exampleforlowboundofA}, illustrates that the decomposition in Theorem \ref{thm:YoeurpdecUMD} takes place only in the UMD space case.

\begin{theorem}\label{thm:exampleforlowboundofA^qA^a}
 Let $X$ be a finite dimensional Banach space, $p\in(1,\infty)$, $\delta\in\bigl(0, \frac{\beta_{p, X}-1}{2}\bigr)$. Then there exist purely discontinuous martingales $M^a, M^q:\mathbb R_+ \times \Omega \to X$ such that $M^a$ has accessible jumps, $M^q$ is quasi-left continuous, $\mathbb E \|M^a_{\infty}\|^p$, $\mathbb E \|M^q_{\infty}\|^p<\infty$, $M^a_0 = M^q_0=0$, and for $M = M^a + M^q$ and $i\in \{a,q\}$ the following holds
 \begin{equation}\label{eq:exampleforlowboundofA^aA^q}
  (\mathbb E \|M^i_{\infty}\|^p)^{\frac 1p} \geq \Bigl(\frac{\beta_{p, X}-1}{2} - \delta \Bigr) (\mathbb E \|M_{\infty}\|^p)^{\frac 1p}.
 \end{equation}
\end{theorem}

For the proof we will need the following lemma.

\begin{lemma}\label{lem:approxRadbyqlc}
 Let $\eps\in \bigl(0,\frac 12\bigr)$, $p\in (1,\infty)$. Then there exist martingales $M, M^a, M^q:[0,1] \times \Omega \to [-1-\eps, 1+\eps]$ with symmetric distributions such that $M^a$ is a martingale with accessible jumps, $\|M^a_1\|_{L^p(\Omega)} <\eps$, $M^q$ is a quasi-left continuous martingale, $M^q_0=0$ a.s., $M = M^a+M^q$, $\sign M_1$ is a Rademacher random variable and 
 \begin{equation}\label{eq:approxRadbyqlc}
    \|M_1 - \sign M_1\|_{L^p(\Omega)}< \eps.
 \end{equation}
\end{lemma}
\begin{proof}
 Let $N^+, N^-:[0,1]\times \Omega \to \mathbb R$ be independent Poisson processes with the same intensity $\lambda_{\eps}$ such that $\mathbb P(N^+_1=0)=\mathbb P(N^-_1=0) <\frac {\eps^p} {2^p}$ (such $\lambda_{\eps}$ exists since $N^+_1$ and $N^-_1$ have Poisson distributions, see \cite{KingPois}). Define a stopping time $\tau$ in the following way:
 \[
  \tau = \inf\{t: N^+_t \geq 1\} \wedge\inf\{t: N^-_t \geq 1\} \wedge 1.
 \]
Let $M^q_t := N^+_{t\wedge \tau} - N^-_{t\wedge \tau}$, $t\in [0,1]$. Then $M^q$ is quasi-left continuous with a symmetric distribution. Let $r$ be an independent Rademacher variable, $M^a_t =\frac {\eps}{2} r$ for each $t\in [0,1]$. Then $M^a$ is a martingale with accessible jumps and symmetric distribution, and $\|M^a_1\|_{L^p(\Omega)} = \frac {\eps}{2}< \eps$. Let $M = M^a + M^q$. Then a.s. 
\begin{equation}\label{eq:lem:approxRadbyqlc}
M_1 \in \Bigl\{-1-\frac{\eps}{2}, -1+\frac{\eps}{2}, -\frac{\eps}{2},\frac{\eps}{2}, 1-\frac{\eps}{2}, 1+\frac{\eps}{2}\Bigr\},
\end{equation}
so $\mathbb P(M_1=0)=0$, and therefore $\sign M_1$ is a Rademacher random variable. Let us prove \eqref{eq:approxRadbyqlc}. Notice that due to \eqref{eq:lem:approxRadbyqlc} if $|M^q_1| = 1$, then $ |M_1 - \sign M_1| < \frac {\eps}{2}$, and if $|M^q_1| = 0$, then $ |M_1 - \sign M_1| <1$. Therefore
\begin{align*}
\mathbb E |M_1 - \sign M_1|^p &= \mathbb E |M_1 - \sign M_1|^p \mathbf 1_{|M^q_1| = 1} +  \mathbb E |M_1 - \sign M_1|^p \mathbf 1_{|M^q_1| = 0}\\
&< \frac {\eps^p} {2^p} + \frac {\eps^p} {2^p} < \eps^p,
\end{align*}
so \eqref{eq:approxRadbyqlc} holds.
\end{proof}
\begin{proof}[Proof of Theorem \ref{thm:exampleforlowboundofA^qA^a}]
 The proof is analogous to the proof of Theorem \ref{thm:exampleforlowboundofA}, while one has to use Lemma \ref{lem:approxRadbyqlc} instead of Lemma \ref{lem:contapproxradem}. 
\end{proof}

Theorem \ref{thm:exampleforlowboundofA^qA^a} yields the following characterization of the UMD property.

\begin{theorem}\label{thm:M^aM^qM^c}
 Let $X$ be a Banach space. Then $X$ is a UMD Banach space if and only if for some (equivalently, for all) $p\in(1,\infty)$ there exists $c_{p,X} >0$ such that for any $L^p$-martingale $M:=\mathbb R_+ \times \Omega \to X$ there exist unique martingales $M^c, M^q, M^a:\mathbb R_+ \times \Omega \to X$ such that $M^c_0=M^q_0=0$, $M^c$ is continuous, $M^q$ is purely discontinuous quasi-left continuous, $M^a$ is purely discontinuous with accessible jumps, $M=M^c + M^q + M^a$, and 
 \begin{equation}\label{eq:coriffXUMDcqlcaj}
    (\mathbb E \|M^c_{\infty}\|^p)^{\frac 1p}+ (\mathbb E \|M^q_{\infty}\|^p)^{\frac 1p} + (\mathbb E \|M^a_{\infty}\|^p)^{\frac 1p} \leq c_{p,X}(\mathbb E \|M_{\infty}\|^p)^{\frac 1p}.
 \end{equation}
If this is the case, then the least admissible $c_{p,X}$ is in the interval $\bigl[\frac{3\beta_{p, X}\!-\!3}{2}\vee 1,3\beta_{p, X}\bigr]$.
\end{theorem}

The decomposition $M =M^c + M^q + M^a$ is called the {\em canonical decomposition} of the martingale $M$ (see \cite{Kal,Yoe76,DY17}).

\begin{proof}
The ``if and only if'' part follows from Theorem \ref{thm:charofUMDbyMeyYoe}, Theorem \ref{thm:YoeurpdecUMD} and Theorem \ref{thm:exampleforlowboundofA^qA^a}. The estimate $c_{p,X} \leq 3\beta_{p, X}$ follows from \eqref{eq:Meyer-Yoeurp} and \eqref{eq:thm:YoeurpdecUMDaq}. The estimate $c_{p,X} \geq \frac{3\beta_{p, X}-3}{2}\!\vee\!1$ follows from \eqref{eq:exampleforlowboundofA} and \eqref{eq:exampleforlowboundofA^aA^q}.
\end{proof}

\begin{corollary}
 Let $X$ be a Banach space. Then $X$ is a UMD Banach space if and only if $\mathcal M_X^{p, d} = \mathcal M_X^{p,a} \oplus\mathcal M_X^{p,q}$ and $\mathcal M_X^p = \mathcal M_X^{p,c}\oplus \mathcal M_X^{p,q}\oplus \mathcal M_X^{p,a}$ for any filtration that satisfies the usual conditions.
\end{corollary}

\begin{proof}
 The corollary follows from Theorem \ref{thm:YoeurpdecUMD}, Theorem \ref{thm:exampleforlowboundofA^qA^a} and Theorem \ref{thm:M^aM^qM^c}.
\end{proof}

\subsection{Stochastic integration}
The current subsection is devoted to application of Theorem \ref{thm:M^aM^qM^c} to stochastic integration with respect to a general martingale.
\begin{proposition*}\label{prop:intpreserves}
 Let $H$ be a Hilbert space, $X$ be a Banach space, $M:\mathbb R_+ \times \Omega \to H$ be a martingale, $\Phi:\mathbb R_+\times \Omega \to \mathcal L(H, X)$ be elementary progressive. Then
 \begin{itemize}
  \item [(i)] if $M$ is continuous, then $\Phi \cdot M$ is continuous;
  \item[(ii)] if $M$ is purely discontinuous, then $\Phi\cdot M$ is purely discontinuous;
  \item[(iii)] if $M$ has accessible jumps, then $\Phi\cdot M$ has accessible jumps;
  \item[(iv)] if $M$ is quasi-left continuous, then $\Phi\cdot M$ is quasi-left continuous.
 \end{itemize}
\end{proposition*}

\begin{proposition}\label{prop:decHilSpacase}
 Let $H$ be a Hilbert space, $M:\mathbb R_+ \times \Omega \to H$ be a local martingale. Then there exist unique martingales $M^c, M^q, M^a:\mathbb R_+ \times \Omega \to H$ such that $M^c$~is continuous, $M^q$ and $M^a$ are purely discontinuous, $M^q$ is quasi-left continuous, $M^a$~has accessible jumps, $M^c_0 = M^q_0=0$ a.s., and $M = M^c + M^q + M^a$.
\end{proposition}

\begin{proof}
Analogously to Theorem 26.14 and Corollary 26.16 in \cite{Kal}.
\end{proof}

\begin{theorem}\label{thm:stochintgenUMDcase}
 Let $H$ be a Hilbert space, $X$ be a UMD Banach space, $p\in (1,\infty)$, $M:\mathbb R_+ \times \Omega \to H$ be a local martingale, $\Phi:\mathbb R_+ \times \Omega \to \mathcal L(H, X)$ be elementary progressive. Let $M=M^c + M^q + M^a$ be the canonical decomposition from Proposition \ref{prop:decHilSpacase}. Then 
 \begin{equation}\label{eq:stochintgenUMDcase}
    \mathbb E\|(\Phi \cdot M)_{\infty}\|^p \eqsim_{p, X}\mathbb E\|(\Phi \cdot M^c)_{\infty}\|^p+ \mathbb E\|(\Phi \cdot M^q)_{\infty}\|^p  + \mathbb E\|(\Phi \cdot M^a)_{\infty}\|^p .
 \end{equation}
 and if $(\Phi \cdot M)_{\infty} \in L^p(\Omega; X)$, then $\Phi \cdot M= \Phi \cdot M^c +  \Phi \cdot M^q +  \Phi \cdot M^a$ is the canonical decomposition from Theorem \ref{thm:M^aM^qM^c}.
\end{theorem}

\begin{proof}
 The statement that $\Phi \cdot M= \Phi \cdot M^c +  \Phi \cdot M^q +  \Phi \cdot M^a$ is the canonical decomposition follows from Proposition \ref{prop:intpreserves}, Theorem \ref{thm:M^aM^qM^c} and the fact that a.s.\ $(\Phi \cdot M)_0 = (\Phi \cdot M^c)_0 = (\Phi \cdot M^q)_0=0$. \eqref{eq:stochintgenUMDcase} follows then from \eqref{eq:coriffXUMDcqlcaj} and the triangle inequality.
\end{proof}

\begin{remark}
 Notice that the It\^o isomorphism for the term $\Phi \cdot M^c$ from \eqref{eq:stochintgenUMDcase} was explored in \cite{VY2016}. It remains open what to do with the other two terms, but positive results in this direction were obtained in the case of $X = L^q(S)$ in \cite{DY17}.
\end{remark}

\section{Weak differential subordination and general martingales}\label{sec:weakdifsubandgenmart}
This subsection is devoted to the generalization of the main theorem in work \cite{Y17FourUMD}. Namely, here we show the $L^p$-estimates for general $X$-valued weakly differentially subordinated martingales.

\begin{definition}
 Let $X$ be a Banach space, $M, N:\mathbb R_+ \times \Omega \to X$ be local martingales. Then $N$ is {\em weakly differentially subordinated} to $M$ if $[\langle M, x^* \rangle] - [\langle N, x^* \rangle]$ is an increasing process a.s.\ for each $x^* \in X^*$.
\end{definition}

The following theorem have been proven in \cite{Y17FourUMD}.
\begin{theorem}\label{thm:ewakdiffsubUMDpurdisc}
 Let $X$ be a Banach space. Then $X$ has the UMD property if and only if for some (equivalently, for all) $p\in (1, \infty)$ there exists $\beta>0$ such that for each pair of purely discontinuous martingales $M, N:\mathbb R_+ \times \Omega\to X$ such that $N$ is weakly differentially subordinated to $M$ one has that
 \[
  (\mathbb E \|N_{\infty}\|^p)^{\frac 1p} \leq \beta (\mathbb E \|M_{\infty}\|^p)^{\frac 1p}.
 \]
If this is the case, then the least admissible $\beta$ is the UMD constant $\beta_{p, X}$.
\end{theorem}

The main goal of the current section is to prove the following generalization of Theorem~\ref{thm:ewakdiffsubUMDpurdisc} to the case of arbitrary martingales.

\begin{theorem}\label{thm:weakdiffsubprocmaingen}
 Let $X$ be a UMD Banach space, $M, N:\mathbb R_+ \times \Omega \to X$ be two martingales such that $N$ is weakly differentially subordinated to $M$. Then for each $p\in (1,\infty)$, $t\geq 0$,
 \begin{equation}\label{eq:weakdiffsubprocmaingen}
  (\mathbb E \|N_t\|^p)^{\frac 1p} \leq \beta_{p,X}^2 (\beta_{p,X}+1)(\mathbb E \|M_t\|^p)^{\frac 1p}.
 \end{equation}
\end{theorem}

The proof will be done in several steps. First we show an analogue of Theorem~\ref{thm:ewakdiffsubUMDpurdisc} for continuous martingales.

\begin{theorem*}\label{thm:weakdiffsubproccont}
 Let $X$ be a Banach space. Then $X$ is a UMD Banach space if and only if for some (equivalently, for all) $p\in (1,\infty)$ there exists $c>0$ such that for any continuous martingales $M, N:\mathbb R_+ \times \Omega \to X$ such that $N$ is weakly differentially subordinated to $M$, $M_0 = N_0 = 0$, one has that
 \begin{equation}\label{eq:weakdiffsubprocmaingencont}
  (\mathbb E \|N_{\infty}\|^p)^{\frac 1p} \leq c_{p,X}(\mathbb E \|M_{\infty}\|^p)^{\frac 1p}.
 \end{equation}
 If this is the case, then the least admissible $c_{p,X}$ is in the segment $[\beta_{p, X}, \beta_{p, X}^2]$.
\end{theorem*}

For the proof we will need the following proposition, which demonstrates that one needs a slightly weaker assumption rather then in Theorem \ref{thm:weakdiffsubproccont} so that the estimate \eqref{eq:weakdiffsubprocmaingencont} holds in a UMD Banach space.

\begin{proposition}\label{prop:weakversofweak}
 Let $X$ be a UMD Banach space, $1<p<\infty$, $M,N:\mathbb R_+ \times \Omega \to X$ be continuous $L^p$-martingales s.t.\ $M_0=N_0=0$ and for each $x^* \in X^*$ a.s.\ for each $t\geq 0$
 \begin{equation}\label{eq:weakversofweak1}
  [\langle N, x^*\rangle]_t \leq [\langle M, x^*\rangle]_t.
 \end{equation}
Then for each $t\geq 0$
\begin{equation}\label{eq:weakversofweak2}
 (\mathbb E \|N_t\|^p)^{\frac 1p} \leq  \beta_{p, X}^2(\mathbb E \|M_t\|^p)^{\frac 1p}.
\end{equation}
\end{proposition}

\begin{proof}
 Without loss of generality by a stopping time argument we assume that $M$ and $N$ are bounded and that $M_{\infty} = M_t$ and $N_{\infty} = N_t$. 

One can also restrict to a finite dimensional case. Indeed, since $X$ is a 
separable reflexive space, $X^*$ is separable as well. Let $(Y_m)_{m\geq 1}$ be 
an increasing sequence of finite-dimensional subspaces of $X^*$ such that 
$\overline{\bigcup_mY_m}=X^*$ and $\|\cdot\|_{Y_m} = \|\cdot\|_{X^*|_{Y_m}}$ for each 
$m\geq 1$. Then for each fixed $m\geq 1$ there exists a linear operator 
$P_m:X\to Y_m^*$ of norm $1$ defined as follows: $\langle P_mx, y\rangle = 
\langle x,y\rangle$ for each $x\in X, y\in Y_m$. Therefore $P_mM$ and $P_m N$ are $Y_m^*$-valued martingales. Moreover, \eqref{eq:weakversofweak1} holds for $P_m M$ and $P_m N$ since there exists $P_m^*:Y_m \to X^*$, and for each $y\in Y_m$ we have that $\langle P_m M, y\rangle = \langle M, P_m y\rangle $ and $\langle P_m N, y\rangle = \langle N, P_m y\rangle$. Since $Y_m$ is a closed 
subspace of $X^*$, \cite[Proposition 4.2.17]{HNVW1} yields $\beta_{p',Y_m}\leq 
\beta_{p',X^*}$, consequently again by \cite[Proposition 4.2.17]{HNVW1} 
$\beta_{p,Y_m^*}\leq \beta_{p,X^{**}} = \beta_{p,X}$. So if we prove the finite dimensional version, then
\[
 (\mathbb E \|P_m N_t\|^p)^{\frac 1p} \leq \beta_{p, Y_m^*}^2(\mathbb E \|P_m M_t\|^p)^{\frac 1p}\leq \beta_{p, X}^2(\mathbb E \|P_m M_t\|^p)^{\frac 1p},
\]
and \eqref{eq:weakversofweak2} with $c_{p,X} = \beta_{p, X}^2$ will follow by letting $m\to \infty$.

Let $d$ be the dimension of $X$, $\vertiii{\cdot}$ be a Euclidean norm on $X\times X$. Let $L= (M,N):\mathbb R_+ \times \Omega \to X\times X$ be a continuous martingale. Since $(X\times X, \vertiii{\cdot})$ is a Hilbert space, $L$ has a continuous quadratic variation $[L]:\mathbb R_+ \times \Omega \to \mathbb R_+$ (see Remark \ref{rem:qvcont}). Let $A:\mathbb R_+ \times \Omega \to \mathbb R_+$ be such that $A_s = [L]_s+s$ for each $s\geq 0$.
Then $A$ is continuous strictly increasing predictable. Define a random time-change $(\tau_s)_{s\geq 0}$ as in Theorem~\ref{thm:apptimechange}.
Let $\mathbb G = (\mathcal G_s)_{s \geq 0} = (\mathcal F_{\tau_s})_{s\geq 0}$ be the induced filtration. Then thanks to the Kazamaki theorem \cite[Theorem 17.24]{Kal} $\widetilde L = L \circ \tau$ is a $G$-martingale, and $[\widetilde L] = [L]\circ \tau$. Notice that $\widetilde L = (\widetilde M, \widetilde N)$ with $\widetilde M = M\circ \tau$, $\widetilde N = N \circ \tau$, and since by Kazamaki theorem \cite[Theorem 17.24]{Kal} $[M\circ \tau] = [M]\circ \tau$, $[N\circ \tau] = [N]\circ \tau$, and $(M\circ \tau)_0 = (N\circ \tau)_0=0$, we have that by \eqref{eq:weakversofweak1} for each $x^*\in X^*$ a.s.\ for each $s\geq 0$
\begin{equation}\label{eq:weakversofweak3}
 [\langle \widetilde N, x^*\rangle]_s = [\langle N, x^*\rangle]_{\tau_s} \leq [\langle M, x^*\rangle]_{\tau_s} = [\langle \widetilde M, x^*\rangle]_s
\end{equation}
Moreover, for all $0\leq u<s$ we have that a.s.
\begin{align*}
 [\widetilde L]_s -  [\widetilde L]_u = ([L]\circ \tau)_s -  ([L]\circ \tau)_u&\leq ([L]\circ \tau)_s + \tau_s -  ([L]\circ \tau)_u - \tau_u\\
 &= ([L]_{\tau_s} + \tau_s) - ([L]_{\tau_u} + \tau_u) = s-u.
\end{align*}
Therefore $[\widetilde L]$ is a.s.\ absolutely continuous with respect to the Lebesgue measure on~$\mathbb R_+$. Consequently, due to Theorem \ref{thm:Brrepres}, there exists an enlarged probability space $(\widetilde {\Omega}, \widetilde {\mathcal F}, \widetilde {\mathbb P})$ with an enlarged filtration $\widetilde {\mathbb G} = (\widetilde{\mathcal G}_s)_{s \geq 0}$, a $2d$-dimensional standard Wiener process $W$, which is defined on~$\widetilde{\mathbb G}$, and a stochastically integrable progressively measurable function $f:\mathbb R_+ \times\widetilde{\Omega} \to \mathcal L(\mathbb R^{2d}, X\times X)$ such that $\widetilde L = f\cdot W$. Let $f^M, f^N:\mathbb R_+ \times \Omega \to \mathcal L(\mathbb R^{2d},X)$ be such that $f = (f^M, f^N)$. Then $\widetilde M = f^M\cdot W$ and $\widetilde N = f^N\cdot W$. Let $(\overline{\Omega}, \overline{\mathcal F}, \overline{\mathbb P})$ be an independent probability space with a filtration~$\overline{\mathbb G}$ and a $2d$-dimensional Wiener process $\overline W$ on it. Denote by $\overline {\mathbb E}$ the expectation on $(\overline{\Omega}, \overline{\mathcal F}, \overline{\mathbb P})$. Then because of the decoupling theorem \cite[Theorem 4.4.1]{HNVW1}, for each~$s\geq 0$
\begin{equation}\label{eq:weakdiffsubproccontdecopl}
 \begin{split}
  (\mathbb E \|\widetilde N_s\|^p)^{\frac 1p} = (\mathbb E \|(f^N\cdot W)_s\|^p)^{\frac 1p}\leq \beta_{p,X}(\mathbb E\, \overline{\mathbb E} \|(f^N\cdot \overline W)_s\|^p)^{\frac 1p},\\
  \frac{1}{\beta_{p,X}}(\mathbb E\, \overline{\mathbb E} \|(f^M\cdot \overline W)_s\|^p)^{\frac 1p} \leq (\mathbb E \|(f^M\cdot W)_s\|^p)^{\frac 1p} = (\mathbb E \|\widetilde M_s\|^p)^{\frac 1p}.
 \end{split}
\end{equation}
Due to the multidimensional version of \cite[Theorem 17.11]{Kal} and \eqref{eq:weakversofweak3} for each $x^* \in X^*$ we have that
\begin{equation}\label{eq:weakdiffsubproccontomega_0}
 s\mapsto [\langle \widetilde M, x^*\rangle]_s - [\langle \widetilde N, x^*\rangle]_s = \int_0^s(|\langle x^*, f^M(r)\rangle|^2 - |\langle x^*, f^N(r)\rangle|^2)\ud r
\end{equation}
is nonnegative and absolutely continuous a.s. Since $X$ is separable, we can fix a set $\widetilde {\Omega}_0\subset \widetilde{\Omega}$ of full measure on which the function \eqref{eq:weakdiffsubproccontomega_0} is nonnegative for each $s\geq 0$.

Now fix $\omega \in \widetilde {\Omega}_0$ and $s\geq 0$. Let us prove that 
$$
\overline{\mathbb E} \|(f^N(\omega)\cdot \overline W)_s\|^p\leq \overline{\mathbb E} \|(f^M(\omega)\cdot \overline W)_s\|^p.
$$
Since $f^M(\omega)$ and $f^N(\omega)$ are deterministic on $\overline {\Omega}$, and since due to \eqref{eq:weakdiffsubproccontomega_0} for each $x^* \in X^*$
\begin{align*}
  \overline{\mathbb E}|\langle (f^N(\omega)\cdot \overline W)_s, x^*\rangle|^2 &= \int_0^s|\langle x^*, f^N(r,\omega)\rangle|^2\ud r\\
  &\leq  \int_0^s|\langle x^*, f^M(r,\omega)\rangle|^2\ud r =  \overline{\mathbb E}|\langle (f^M(\omega)\cdot \overline W)_s, x^*\rangle|^2,
\end{align*}
by \cite[Corollary 4.4]{NW1} we have that $\overline{\mathbb E} \|(f^N(\omega)\cdot \overline W)_s\|^p\leq \overline{\mathbb E} \|(f^M(\omega)\cdot \overline W)_s\|^p$. Consequently, due to~\eqref{eq:weakdiffsubproccontdecopl} and the fact that $\widetilde{\mathbb P}(\Omega_0)=1$
\begin{align*}
 (\mathbb E \|\widetilde N_s\|^p)^{\frac 1p} \leq \beta_{p,X}(\mathbb E\, \overline{\mathbb E} \|(f^N\cdot \overline W)_s\|^p)^{\frac 1p}&\leq \beta_{p,X}(\mathbb E\, \overline{\mathbb E} \|(f^M\cdot \overline W)_s\|^p)^{\frac 1p}\leq \beta_{p,X}^2(\mathbb E \|\widetilde M_s\|^p)^{\frac 1p}.
\end{align*}
Recall that $\widetilde M$ and $\widetilde N$ are bounded, so thanks to the dominated convergence theorem one gets \eqref{eq:weakversofweak2} with $c_{p,X} = \beta_{p, X}^2$ by letting $s$ to infinity.
\end{proof}

\begin{proof}[Proof of Theorem \ref{thm:weakdiffsubproccont}] 
{\em The ``only if'' part \& the upper bound of $c_{p,X}$:} The ``only if'' part and the estimate $c_{p,X}\leq \beta_{p, X}^2$ follows from Proposition \ref{prop:weakversofweak} since \eqref{eq:weakversofweak1} holds for $M$ and $N$ because $N$ is weakly differentially subordinated to $M$.

{\em The ``if'' part \& the lower bound of $c_{p,X}$:} {\em See the supplement \cite{Y17MDSupp}.}

\end{proof}

\begin{remark}\label{rem:Hiltranfweakdiffsubcont}
Let $X$ be a Banach space. Then according to \cite{Bour83,Burk83,Gar85} the Hilbert transform $\mathcal H_X$ can be extended to $L^p(\mathbb R; X)$ for each $1<p<\infty$ if and only if $X$ is a UMD Banach space. Moreover, if this is the case, then 
\[
 \sqrt{\beta_{p, X}}\leq\|\mathcal H_X\|_{\mathcal L(L^p(\mathbb R; X))}\leq \beta_{p, X}^2.
\]
As it was shown in \cite{Y17FourUMD}, the upper bound $\beta_{p, X}^2$ can be also directly derived from the upper bound for $c_{p,X}$ in Theorem \ref{thm:weakdiffsubproccont}. The sharp upper bound for $\|\mathcal H_X\|_{\mathcal L(L^p(\mathbb R; X))}$ remains an open question (see \cite[pp. 496-497]{HNVW1}), so the sharp upper bound for $c_{p,X}$ is of interest. 
\end{remark}

\begin{lemma*}\label{lemma:weakdiffsubdecomp}
 Let $X$ be a Banach space, $M^c,N^c:\mathbb R_+ \times \Omega \to X$ be continuous martingales, $M^d, N^d:\mathbb R_+ \times \Omega \to X$ be purely discontinuous martingales, $M^c_0 = N^c_0=0$. Let $M := M^c +M^d$, $N:= N^c + N^d$. Suppose that $N$ is weakly differentially subordinated to $M$. Then $N^c$ is weakly differentially subordinated to $M^c$, and $N^d$ is weakly differentially subordinated to $M^d$.
\end{lemma*}

\begin{proof}[Proof of Theorem \ref{thm:weakdiffsubprocmaingen}]
 By Theorem \ref{thm:Meyer-Yoeurp} there exist martingales $M^d, M^c, N^d, N^c:\mathbb R_+ \times \Omega \to X$ such that $M^d$ and $N^d$ are purely discontinuous, $M^c$ and $N^c$ are continuous, $M^c_0 = N^c_0=0$, and $M = M^d+ M^c$ and $N = N^d+ N^c$. By Lemma~\ref{lemma:weakdiffsubdecomp}, $N^d$ is weakly differentially subordinated to $M^d$ and $N^c$ is weakly differentially subordinated to $M^c$. Therefore for each $t\geq 0$
 \begin{align*}
  (\mathbb E \|N_t\|^p)^{\frac 1p} \stackrel{(i)}\leq (\mathbb E \|N^d_t\|^p)^{\frac 1p} + (\mathbb E \|N^c_t\|^p)^{\frac 1p}
 &\stackrel{(ii)} \leq \beta_{p, X}^2(\mathbb E \|M^d_t\|^p)^{\frac 1p} + \beta_{p, X}(\mathbb E \|M^c_t\|^p)^{\frac 1p}\\
  & \stackrel{(iii)}\leq\beta_{p, X}^2(\beta_{p,X}+1)(\mathbb E \|M_t\|^p)^{\frac 1p}, 
 \end{align*}
where $(i)$ holds thanks to the triangle inequality, $(ii)$ follows from Theorem \ref{thm:ewakdiffsubUMDpurdisc} and Theorem \ref{thm:weakdiffsubproccont}, and $(iii)$ follows from \eqref{eq:Meyer-Yoeurp}.
\end{proof}

\begin{remark}
 It is worth noticing that in a view of recent results the sharp constant in \eqref{eq:Meyer-Yoeurp} and \eqref{eq:thm:YoeurpdecUMDaq} can be derived and equals the {\em UMD$^{\{0,1\}}_p$ constant} $\beta_{p, X}^{\{0,1\}}$. In order to show that this is the right upper bound one needs to use a {\em $\{0,1\}$-Burkholder function} instead of the Burkholder function, while the sharpness follows analogously Theorem \ref{thm:exampleforlowboundofA} and \ref{thm:exampleforlowboundofA^qA^a}. See \cite{Y17UMD^A} for details.
\end{remark}

\begin{remark}
 In the recent paper \cite{Y17GMY} the existence of the canonical decomposition of a general local martingale together with the corresponding weak $L^1$-estimates were shown. Again existence of the canonical decomposition of any $X$-valued martingale is equivalent to $X$ having the UMD property.
\end{remark}

\section*{Acknowledgements} The author would like to thank Mark Veraar for helpful comments; in particular for showing him \cite[Corollary 4.4]{NW1}. The author thanks Jan van Neerven for careful reading of parts of this article and useful suggestions. The author thanks the anonymous referee for his/her valuable comments.

\bibliographystyle{plain}

\def\cprime{$'$} \def\polhk#1{\setbox0=\hbox{#1}{\ooalign{\hidewidth
  \lower1.5ex\hbox{`}\hidewidth\crcr\unhbox0}}}
  \def\polhk#1{\setbox0=\hbox{#1}{\ooalign{\hidewidth
  \lower1.5ex\hbox{`}\hidewidth\crcr\unhbox0}}} \def\cprime{$'$}

\newpage
\renewcommand\thesection{\Alph{section}}
\setcounter{section}{18}
\section{Supplement: Some proofs}

\begin{customprop}{\ref{prop:dualofMXp}}
  Let $X$ be a Banach space with the Radon-Nikod\'ym property (e.g.\ reflexive), $1<p<\infty$. Then $(\mathcal M_X^p)^* = \mathcal M_{X^*}^{p'}$, and $\|M\|_{(\mathcal M_X^p)^*} = \|M\|_{\mathcal M_{X^*}^{p'}}$ for each $M \in \mathcal M_{X^*}^{p'}$.
\end{customprop}
\begin{proof}
 Since $\|M\|_{\mathcal M_X^p} = \|M_{\infty}\|_{L^p(\Omega; X)}$ for each $M \in \mathcal M_X^p$, and since for each $\xi\in L^p(\Omega; X)$ we can construct a martingale $M = (M_t)_{t\geq 0} = (\mathbb E(\xi|\mathcal F_t))_{t\geq 0}$ such that $\|M\|_{\mathcal M_X^p} = \|\xi\|_{L^p(\Omega; X)}$, $\mathcal M_X^p$ is isometric to $L^p(\Omega; X)$, and therefore the proposition follows from \cite[Proposition 1.3.3]{HNVW1}.
\end{proof}

\begin{customprop}{\ref{thm:purdiscorthtoanycont1}}
   A martingale $M:\mathbb R_+\times \Omega \to \mathbb R$ is purely discontinuous if and only if $M N$ is a martingale for any continuous bounded martingale $N:\mathbb R_+ \times \Omega \to \mathbb R$ with $N_0 = 0$. 
\end{customprop}

   \begin{proof}
  One direction follows from \cite[Corollary 26.15]{Kal}. Indeed, if $M$ is purely discontinuous, then a.s.\ $[M,N] = 0$. Therefore by Remark \ref{rem:covariation}, $MN$ is a local martingale, and due to integrability it is a martingale. 
 
 For the other direction we apply Remark \ref{rem:YoeMeyR}. Let $N:\mathbb R_+ \times \Omega \to \mathbb R$ be a continuous martingale such that $N_0=0$ and $M - N$ is purely discontinuous. Then there exists an increasing sequence of stopping times $(\tau_n)_{n\geq 1}$ such that $\tau_n \nearrow \infty$ as $n\to \infty$ and $N^{\tau_n}$ is a bounded continuous martingale for each $n\geq 1$. Therefore $MN^{\tau_n}$ and $(M-N)N^{\tau_n}$ are martingales for any $n\geq 1$, and hence $(N^{\tau_n})^2 = (MN^{\tau_n} - (M-N)N^{\tau_n})^{\tau_n}$ is a martingale that starts at zero. On the other hand it is a nonnegative martingale, so it is the zero martingale. By letting $n$ to infinity we prove that $N=0$ a.s., so $M$ is purely discontinuous.
 \end{proof}
 
 \begin{customthm}{\ref{thm:apptimechange}}
  Let $A:\mathbb R_+ \times \Omega\to \mathbb R_+$ be a strictly increasing continuous predictable process such that $A_0 = 0$ and $A_{t} \to \infty$ as $t\to \infty$ a.s. Let $\tau = (\tau_s)_{s\geq 0}$ be a random time-change defined as
  $\tau_s := \{t: A_t=s\}$, $s\geq 0$.
Then $(A\circ \tau)(t) = (\tau \circ A)(t) = t$ a.s.\ for each $t\geq 0$. Let $\mathbb G = (\mathcal G_s)_{s \geq 0} = (\mathcal F_{\tau_s})_{s\geq 0}$ be the induced filtration. Then $(A_t)_{t\geq 0}$ is a random time-change with respect to $\mathbb G$ and for any $\mathbb F$-martingale $M:\mathbb R_+ \times \Omega \to\mathbb R$ the following holds
\begin{itemize}
 \item [(i)] $M \circ \tau$ is a continuous $\mathbb G$-martingale if and only if $M$ is continuous, and
 \item[(ii)]$M \circ \tau$ is a purely discontinuous $\mathbb G$-martingale if and only if $M$ is purely discontinuous.
\end{itemize}
 \end{customthm}

 \begin{proof}
Let us first show that $(A\circ \tau)(t) = (\tau \circ A)(t) = t$ a.s.\ for each $t\geq 0$. Fix $t\geq 0$. Then a.s.
\begin{equation}\label{eq:apptimechange1}
  (\tau\circ A)(t) = \tau_{A_t} = \{s: A_s = A_t\} = t.
\end{equation}
Since $A$ is strictly increasing continuous and starts at zero, there exists $S_t:\Omega \to \mathbb R_+$ such that $A_{S_t} = t$ a.s. Then by \eqref{eq:apptimechange1} and the definition of $S_t$ a.s.
\[
 (A\circ \tau)(t) = (A\circ \tau)(A_{S_t}) = (A\circ (\tau\circ A))(S_t) = A_{S_t} = t.
\]

Now we turn to the second part of the theorem. Notice that $s\mapsto \tau_s$, $s\geq 0$, is a continuous strictly increasing $\mathbb G$-predictable process which starts at zero. Then for each $t\geq 0$ one has that $A_t = \{s:\tau_s = t\}$, so $(A_t)_{t\geq 0}$ is a random time-change with respect to the filtration~$\mathbb G$. Since $(A\circ \tau)(t) = (\tau \circ A)(t) = t$ a.s.\ for each $t\geq 0$, it is sufficient to show only ``if'' parts of both (i) and (ii).

 (i) follows from the fact that $\tau_{s-} = \tau_s$ (so $M$ is $\tau$-continuous), and the Kazamaki theorem \cite[Theorem 17.24]{Kal}. Let us now show (ii). Thanks to \cite[Theorem 7.12]{Kal} $M \circ \tau$ is a martingale.  Let $N:\mathbb R_+ \times \Omega \to \mathbb R$ be a continuous bounded $\mathbb G$-martingale such that $N_0 = 0$. Then by (i), $N\circ A$ is a continuous bounded $\mathbb F$-martingale, and therefore by Proposition \ref{thm:purdiscorthtoanycont1} the process $M \cdot(N\circ A)$ is a martingale. Consequently due to \cite[Theorem 7.12]{Kal}, $(M\circ \tau)N = (M \cdot(N\circ A))\circ \tau$ is a martingale. Since $N$ is taken arbitrary and due to Proposition \ref{thm:purdiscorthtoanycont1}, $M\circ \tau$~is purely discontinuous.
\end{proof}

\begin{customlemma}{\ref{lemma:traceindep}}
  Let $d$ be a natural number, $E$ be a $d$-dimensional linear space. Let $V: E\times E \to \mathbb R$ and $W:E^* \times E^* \to \mathbb R$ be two bilinear functions. Then the expression
 \begin{equation}\label{eq:traceindep}
  \sum_{n, m=1}^d V(e_n, e_m) W(e_n^*, e_m^*)
 \end{equation}
does not depend on the choice of basis $(e_n)_{n=1}^d$ of $E$ (here $(e_n^*)_{n=1}^d$ is the corresponding dual basis of $(e_n)_{n=1}^d$).
\end{customlemma}

\begin{proof}
Let $(e_n)_{n=1}^d$ be a basis of $E$, $(e_n^*)_{n=1}^d$ be the corresponding dual basis. Fix another basis $(\tilde e_n)_{n=1}^d$ of $E$. Let $(\tilde e_n^*)_{n=1}^d$ be the corresponding dual basis of $E^*$. Let matrices $A = (a_{ij})_{i,j=1}^d$ and $B = (b_{ij})_{i,j=1}^d$ be such that $\tilde e_n =\sum_{i=1}^d a_{ni}e_i$, $\tilde e^*_n = \sum_{i=1}^db_{ni}e_i^*$ for each $n=1,\ldots, d$. Then for each $n,m = 1,\ldots,d$
 \[
  \delta_{nm} = \langle \tilde e_n,\tilde e_m^*\rangle = \Bigl\langle \sum_{i=1}^d a_{ni}e_i, \sum_{j=1}^db_{mj}e_j^*  \Bigr\rangle = \sum_{i=1}^da_{ni}b_{mi}.
 \]
Hence $A^T B=I$, and thus also $AB^T=I$ is the identical matrix as well, and therefore $\sum_{i=1}^da_{in}b_{im}=\delta_{nm}$ for each $n,m=1.\ldots,d$. Consequently, if we paste $(\tilde e_n)_{n=1}^d$ and $(\tilde e_n^*)_{n=1}^d$ in \eqref{eq:traceindep}, due to the bilinearity of $V$ and $W$
\begin{align*}
  \sum_{n, m=1}^d V(\tilde e_n, \tilde e_m) W(\tilde e_n^*, \tilde e_m^*) &= \sum_{i,j,k,l,n, m=1}^d V(a_{ni} e_i, a_{mj} e_j) W(b_{nk} e_k^*, b_{ml} e_l^*)\\
  &=\sum_{i,j,k,l=1}^d \sum_{n=1}^da_{ni}b_{nk} \sum_{m=1}^da_{mj}b_{ml}V( e_i,  e_j) W(e_k^*,  e_l^*)\\
  &=\sum_{i,j,k,l=1}^d \delta_{ik} \delta_{jl}V( e_i,  e_j) W(e_k^*,  e_l^*)\\
  &= \sum_{i,j=1}^d V( e_i,  e_j) W(e_i^*,  e_j^*).
\end{align*}
\end{proof}

\begin{customlemma}{\ref{lem:contapproxradem}}
 Let $\eps>0$, $p\in (1,\infty)$. Then there exists a continuous martingale $M:[0,1]\times \Omega \to [-1,1]$ with a symmetric distribution such that $\sign M_1$ is a~Rademacher random variable and
\[
  \|M_1-\sign M_1\|_{L^p(\Omega)}<\eps.\eqno{\eqref{eq:lemcontapproxradem}}
\]
\end{customlemma}

\begin{proof}
 Let $W:[0,1]\times \Omega \to \mathbb R$ be a standard Wiener process. For each $n\geq 1$ we define a stopping time $\tau_n := \inf\{t:|W_t|>\frac 1n\}\wedge1$. Then $\tau_n \to 0$ a.s.\ as $n\to \infty$, and hence there exists $N\geq 1$ such that $\mathbb P(N W^{\tau_N}_1 = \sign W^{\tau_N}_1)> 1-\frac {\eps^p}{2^p}$. Let $M = NW^{\tau_N}$. Then
 \[
  \|M_1-\sign M_1\|_{L^p(\Omega)} \leq \Bigr(\mathbb E \bigl[(|M_1| + 1)^p \mathbf 1_{M_1 \neq \sign M_1}\bigr]\Bigl)^{\frac 1p}< \Bigr(2^p \cdot \frac{\eps^p}{2^p}\Bigl)^{\frac 1p} \leq \eps,
 \]
 and \eqref{eq:lemcontapproxradem} follows. 
 
 Notice that since $W$ is a Wiener process, $W_1$ has a standard Gaussian distribution. Consequently, 
 \[
  \mathbb P(M_1 = 0) = \mathbb P(NW_1^{\tau_N} = 0) \leq \mathbb P(NW_1 = 0) = 0,
 \]
 and since $W^{\tau_N}$ has a symmetric distribution, $\sign M_1$ is Rademacher.
\end{proof}

\begin{customthm}{\ref{thm:charofUMDbyMeyYoe}}
  Let $X$ be a Banach space. Then $X$ is UMD if and only if for some (or, equivalently, for all) $p\in(1,\infty)$, for any probability space $(\Omega, \mathcal F,\mathbb P)$ with any filtration $\mathbb F = (\mathcal F_t)_{t\geq 0}$ that satisfies the usual conditions, $\mathcal M_X^{p} = \mathcal M_X^{p, d} \oplus \mathcal M_X^{p, c}$, and there exist projections $A^d, A^c\in \mathcal L(\mathcal M_X^{p})$ such that $\text{ran } A^d = \mathcal M_X^{p, d}$, $\text{ran } A^c = \mathcal M_X^{p, c}$, and for any $M \in \mathcal M_X^{p}$ the decomposition $M = A^d M + A^cM$
is the Meyer-Yoeurp decomposition from Theorem \ref{thm:Meyer-Yoeurp}. If this is the case, then
\[
  \|A^d\|\leq \beta_{p,X}\;\; \text{and}\;\;
  \|A^c\|\leq  \beta_{p,X}. \eqno{\eqref{eq:charofUMDbyMeyYoe1}}
\]
Moreover, there exist  $(\Omega, \mathcal F,\mathbb P)$ and $\mathbb F = (\mathcal F_t)_{t\geq 0}$ such that
\[
  \|A^d\|,\|A^c\|\geq \frac {\beta_{p, X}-1}{2}\vee 1.  \eqno{\eqref{eq:charofUMDbyMeyYoe2}}
\]
\end{customthm}

\begin{proof}
The ``if'' part follows from \eqref{eq:charofUMDbyMeyYoe1}, and the ``only if'' part follows from \eqref{eq:charofUMDbyMeyYoe2}, so it is sufficient to show \eqref{eq:charofUMDbyMeyYoe1} and \eqref{eq:charofUMDbyMeyYoe2}. \eqref{eq:charofUMDbyMeyYoe1} is equivalent to \eqref{eq:Meyer-Yoeurp}. The bound $\geq \frac {\beta_{p, X}-1}{2}$ in \eqref{eq:charofUMDbyMeyYoe2} follows from Theorem \ref{thm:exampleforlowboundofA}, while the bound $\geq 1$ follows from the fact that both $A^d$ and $A^c$ are projections onto nonzero spaces $\mathcal M^{p, d}_X$ and $\mathcal M^{p, c}_X$ respectively.
\end{proof}

\begin{customprop}{\ref{prop:limthmforaccjqlc}}
 Let $1<p<\infty$, $M:\mathbb R_+ \times \Omega \to \mathbb R$ be a purely discontinuous \mbox{$L^p$-mar}\-tin\-gale. Let $(M^n)_{n\geq 1}$ be a sequence of purely discontinuous martingales such that $M^n_{\infty} \to M_{\infty}$ in $L^p(\Omega)$. Then the following assertions hold
 \begin{itemize}
  \item [(a)] if $(M^n)_{n\geq 1}$ have accessible jumps, then $M$ has accessible jumps as well;
    \item [(b)] if $(M^n)_{n\geq 1}$ are quasi-left continuous martingales, then $M$ is quasi-left continuous as well.
 \end{itemize}
\end{customprop}

\begin{proof}
 We will only show (a), (b) can be proven in the same way. Without loss of generality suppose that $M_0=0$ and $M^n_0 = 0$ for each $n\geq 1$. Let $M^a, M^q:\mathbb R_+ \times \Omega \to \mathbb R$ be purely discontinuous martingales such that $M^a$ has accessible jumps, $M^q$ is quasi-left continuous, $M^a_0=M^q_0=0$ and $M=M^a + M^q$ (see Theorem~\ref{thm:Yoeurpdec}). Then by Theorem \ref{thm:Yoeurpdec}, the Doob maximal inequality \cite[Theorem 1.3.8(iv)]{KS} and the fact the a quadratic variation is a.s.\ nonnegative
 \begin{align*}
   \mathbb E|M_{\infty} - M^n_{\infty}|^p \eqsim_{p} \mathbb E[M - M^n]_{\infty}^{\frac p2}
   =\mathbb E \Bigl([M^a - M^n]_{\infty} + [M^q]_{\infty}\Bigr)^{\frac p2}\geq \mathbb E[M^q]_{\infty}^{\frac p2},
 \end{align*}
 and since $\mathbb E|M_{\infty} - M^n_{\infty}|^p \to 0$ as $n \to \infty$, $\mathbb E[M^q]_{\infty}^{\frac p2} = 0$. Therefore $M^q = 0$ a.s., so $M$ has accessible jumps.
\end{proof}

\begin{customlemma}{\ref{lem:weakdefforaccjumpsandqlc}}
  Let $X$ be a reflexive Banach space, $M:\mathbb R_+ \times \Omega \to X$ be a purely discontinuous martingale.
 \begin{itemize}
  \item [(i)] $M$ has accessible jumps if and only if for each $x^* \in X^*$ the martingale $\langle M, x^*\rangle$ has accessible jumps;
   \item [(ii)] $M$ is quasi-left continuous if and only if for each $x^* \in X^*$ the martingale $\langle M, x^*\rangle$ is quasi-left continuous.
 \end{itemize}
\end{customlemma}

\begin{proof}
Without loss of generality we can assume that $X$ is a separable Banach space. We will show only $(i)$, while $(ii)$ can be proven analogously.

 (i): The ``only if'' part is obvious. For ``if'' part we fix a dense subset $(x_m^*)_{m\geq 1}$ of $X^*$.  Let $\tau$ be a totally inaccessible stopping time. Then $\Delta \langle M_{\tau}, x_m^*\rangle = \langle \Delta M_{\tau}, x_m^*\rangle = 0$ a.s.\ for each $m\geq 1$. Hence $\Delta M_{\tau} = 0$ a.s., and the ``if'' part is proven.
\end{proof}

\begin{customprop}{\ref{prop:M^pqandM^paareclosed}}
  Let $X$ be a Banach space, $1<p<\infty$. Then $\mathcal M_X^{p, q}$ and $\mathcal M_X^{p, a}$ are closed subspaces of $\mathcal M_X^{p, d}$.
\end{customprop}

\begin{proof}
 We only will show the case of $\mathcal M_X^{p, q}$, the proof for $\mathcal M_X^{p, a}$ is analogous. Let $(M^n)_{n\geq 1}\in \mathcal M_X^{p, q}$ be such that $(M^n_{\infty})_{n\geq 1}$ is a Cauchy sequence in $L^p(\Omega; X)$. Let $\xi = \lim_{n\to\infty} M^n_{\infty}$ in $L^p(\Omega; X)$. Define an $X$-valued martingale $M$ as follows: $M_t = \mathbb E (\xi|\mathcal F_t)$, $t\geq 0$. Then since conditional expectation is a contraction in $L^p(\Omega; X)$, $M_0 = \lim_{n\to \infty}M^n_0 = 0$. Now let us show that $M$ is quasi-left continuous. By Lemma \ref{lem:weakdefforaccjumpsandqlc} it is sufficient to show that $\langle M, x^*\rangle$ is quasi-left continuous for each $x^*\in X^*$. Fix $x^* \in X^*$. Define $N:= \langle M, x^*\rangle$ and $N^n:= \langle M^n, x^*\rangle$ for each $n\geq 1$. Then
 \begin{align*}
  \mathbb E\|N_{\infty}-N^n_{\infty}\|^p &\eqsim_p \mathbb E[N-N^n]_{\infty}^{\frac p2} = \mathbb E \bigl([N-N^n]^c_{\infty} + [N-N^n]^q_{\infty}+ [N-N^n]^a_{\infty}\bigr)^{\frac p2}\\
  &=\mathbb E \bigl([N]^c_{\infty} + [N-N^n]^q_{\infty}+ [N]^a_{\infty}\bigr)^{\frac p2}\geq \mathbb E \bigl([N]^c_{\infty}+ [N]^a_{\infty}\bigr)^{\frac p2},
 \end{align*}
and since the first expression vanishes as $n\to \infty$, $[N]^c_{\infty} = [N]^a_{\infty}=0$ a.s., so $N$ is quasi-left continuous. Since $x^*\in X^*$ was arbitrary, $M\in \mathcal M_X^{p, q}$.
\end{proof}

\begin{customlemma}{\ref{lem:accjumps+qlc=cons}}

 Let $X$ be a Banach space, $M:\mathbb R_+ \times\Omega \to X$ be a purely discontinuous martingale. Let $M$ be both with accessible jumps and quasi-left continuous. Then $M = M_0$ a.s. In other words, $\mathcal M_X^{p, q}\cap\mathcal M_X^{p, a} = 0$.
 \end{customlemma}

 \begin{proof}
 Without loss of generality set $M_0=0$. Suppose that $\mathbb P(M\neq 0)>0$. Then there exists $x^* \in X^*$ such that $\mathbb P(\langle M, x^*\rangle \neq0)>0$. Let $N = \langle M, x^*\rangle$. Then $N$ is both with accessible jumps and quasi-left continuous. Hence by Corollary \ref{cor:M=M_0ifMbothaccjandqlc}, $N = 0$ a.s., and therefore $M=0$ a.s.
\end{proof}

\begin{customprop}{\ref{prop:intpreserves}}
  Let $H$ be a Hilbert space, $X$ be a Banach space, $M:\mathbb R_+ \times \Omega \to H$ be a martingale, $\Phi:\mathbb R_+\times \Omega \to \mathcal L(H, X)$ be elementary progressive. Then
 \begin{itemize}
  \item [(i)] if $M$ is continuous, then $\Phi \cdot M$ is continuous;
  \item[(ii)] if $M$ is purely discontinuous, then $\Phi\cdot M$ is purely discontinuous;
  \item[(iii)] if $M$ has accessible jumps, then $\Phi\cdot M$ has accessible jumps;
  \item[(iv)] if $M$ is quasi-left continuous, then $\Phi\cdot M$ is quasi-left continuous.
 \end{itemize}
\end{customprop}

\begin{proof}
 (i): If $M$ is continuous, then by the construction of a stochastic integral \eqref{eq:defofstochintwrtM}, $\Phi \cdot M$ is a finite sum of continuous martingales, so it is continuous as well.
 
 (ii): Notice that according to Remark \ref{rem:purdiscorthtoanycont2} the space of purely discontinuous martingales is linear, so again as in (i) by Proposition \ref{thm:purdiscorthtoanycont1} and \eqref{eq:defofstochintwrtM}, $\Phi \cdot M$ is a finite sum of purely discontinuous martingales, so it is purely discontinuous as well.
 
 (iii) and (iv): By \eqref{eq:defofstochintwrtM} we have that for any stopping time $\tau$ a.s.\ $\Delta(\Phi \cdot M)_{\tau} \neq 0$ implies $\Delta M_{\tau} \neq 0$. 
 Therefore by Definition \ref{def:accjumpsqlc} if $M$ has accessible jumps, then $\Phi \cdot M$ has them as well, and if $M$ is quasi-left continuous, then $\Phi \cdot M$ is quasi-left continuous as well.
\end{proof}

\begin{customthm}{\ref{thm:weakdiffsubproccont}}
  Let $X$ be a Banach space. Then $X$ is a UMD Banach space if and only if for some (equivalently, for all) $p\in (1,\infty)$ there exists $c>0$ such that for any continuous martingales $M, N:\mathbb R_+ \times \Omega \to X$ such that $N$ is weakly differentially subordinated to $M$, $M_0 = N_0 = 0$, one has that
\[
  (\mathbb E \|N_{\infty}\|^p)^{\frac 1p} \leq c_{p,X}(\mathbb E \|M_{\infty}\|^p)^{\frac 1p}. \eqno{\eqref{eq:weakdiffsubprocmaingencont}}
\]
 If this is the case, then the least admissible $c_{p,X}$ is in the segment $[\beta_{p, X}, \beta_{p, X}^2]$.
\end{customthm}

\begin{proof}
{\em The ``if'' part \& the lower bound of $c_{p,X}$:} Let $\beta_{p, X}$ be the UMD constant of $X$ ($\beta_{p, X} = \infty$ if $X$ is not a UMD space). Fix $K\geq 1$. Then by \cite[Theorem 4.2.5]{HNVW1} there exists $N\geq 1$, a Paley-Walsh martingale difference sequence $(d_n)_{n=1}^N$, and a $\{-1,1\}$-valued sequence $(\eps_n)_{n=1}^N$ such that
\[
 \Bigl(\mathbb E \Bigl\|\sum_{n=1}^N \eps_nd_n\Bigr\|^p\Bigr)^{\frac 1p} \geq \Bigl(\beta_{p, X}\wedge 2K-\frac {1}{2K}\Bigr) \Bigl(\mathbb E \Bigl\|\sum_{n=1}^N d_n\Bigr\|^p\Bigr)^{\frac 1p}
\]
Without loss of generality we can assume that 
\[
 \Bigl(\mathbb E \Bigl\|\sum_{n=1}^N \eps_nd_n\Bigr\|^p\Bigr)^{\frac 1p},\Bigl(\mathbb E \Bigl\|\sum_{n=1}^N d_n\Bigr\|^p\Bigr)^{\frac 1p} \leq 1.
\]
Let $(r_n)_{n=1}^N$ be a sequence of Rademacher variables and $(\phi_n)_{n=1}^N$ be a sequence of functions as in Definition \ref{def:Paley-Walsh}, i.e.\ be such that $d_n =  r_n\phi_n(r_1,\ldots,r_{n-1})$ for each $n=1, \ldots,N$.

By the same techniques as were used in the proof of Theorem \ref{thm:exampleforlowboundofA} we can find a sequence of independent continuous real-valued symmetric martingales $(M^n)_{n=1}^N$ on $[0,1]$ such that for each $n=1,\ldots,N$
\begin{equation}\label{eq:approxinthm:weakdiffsubproccont}
 \|(M^n - \sign M^n)\phi_n(\sign M^1,\ldots,\sign M^{n-1})\|_{L^p(\Omega; X)}\leq \frac{1}{8NK^2}.
\end{equation}

Let $\sigma_n = \sign M^n$ for each $n=1,\ldots, N$. Then we define continuous martingales $M, N:\mathbb R_+\times \Omega \to X$ in the following way:
\[
M_t =
\begin{cases}
0, &\text{if}\;\; 0\leq t\leq1;\\
 M_{n}+ M^n_{t-n}\phi_{n}(\sigma_1,\ldots,\sigma_{n-1}),&\text{if}\;\;  t\in(n,n+1], n\in \{1\ldots,N\},\\
  M_{N+1},&\text{if}\;\;  t>N+1,
\end{cases}  
 \]

 \[
N_t =
\begin{cases}
0, &\text{if}\;\; 0\leq t\leq1;\\
 M_{n}+ \eps_nM^n_{t-n}\phi_{n}(\sigma_1,\ldots,\sigma_{n-1}),&\text{if}\;\;  t\in(n,n+1], n\in \{1\ldots,N\},\\
  N_{N+1},&\text{if}\;\;  t>N+1.
\end{cases}  
 \]
 Then $N$ is weakly differentially subordinated to $M$. Indeed, for each $x^* \in X^*$, $n\in \{1,\ldots, N\}$ and $t\in [n ,n+1]$ a.s.
 \begin{align*}
    [\langle M,x^*\rangle]_t - [\langle M,x^*\rangle]_n &= [M^n]_{t-n} |\langle \phi_{n}(\sigma_1,\ldots,\sigma_{n-1}),x^*\rangle|^2\\
    &= [M^n]_{t-n} |\langle \eps_n\phi_{n}(\sigma_1,\ldots,\sigma_{n-1}),x^*\rangle|^2\\
    &=[\langle N,x^*\rangle]_t - [\langle N,x^*\rangle]_n,
 \end{align*}
therefore, since $M_1 = N_1 = 0$ a.s., we have that for each $x^* \in X^*$ and $t\geq 0$ a.s.\ $[\langle M,x^*\rangle]_t = [\langle N,x^*\rangle]_t$, so $N$ is weakly differentially subordinated to $M$. Then
\begin{align*}
 (\mathbb E \|N_{\infty}\|^p)^{\frac 1p} &= \Bigl(\mathbb E\Bigl\|\sum_{n=1}^N \eps_nM^n_1 \phi_{n}(\sigma_1,\ldots,\sigma_{n-1})\Bigr\|  \Bigr)^{\frac 1p}\\
 &\stackrel{(i)}\geq \Bigl(\mathbb E\Bigl\|\sum_{n=1}^N \eps_n\sigma_n \phi_{n}(\sigma_1,\ldots,\sigma_{n-1})\Bigr\|  \Bigr)^{\frac 1p}\\
 &\quad - \sum_{n=1}^N  \|(M^n - \sigma_n)\phi_n(\sigma_1,\ldots,\sigma_{n-1})\|_{L^p(\Omega; X)}\\
 &\stackrel{(ii)}\geq \Bigl(\beta_{p, X}\wedge 2K-\frac {1}{2K}\Bigr)\Bigl(\mathbb E\Bigl\|\sum_{n=1}^N\sigma_n \phi_{n}(\sigma_1,\ldots,\sigma_{n-1})\Bigr\|  \Bigr)^{\frac 1p}- \frac{1}{8K^2}\\
 &\stackrel{(iii)}\geq \Bigl(\beta_{p, X}\wedge 2K-\frac {1}{2K}\Bigr)\Bigl(\mathbb E\Bigl\|\sum_{n=1}^NM^n_1 \phi_{n}(\sigma_1,\ldots,\sigma_{n-1})\Bigr\|  \Bigr)^{\frac 1p}\\
 &\quad - 2K \sum_{n=1}^N  \|(M^n - \sigma_n)\phi_n(\sigma_1,\ldots,\sigma_{n-1})\|_{L^p(\Omega; X)} - \frac{1}{8K^2}\\
 &\stackrel{(iv)}\geq \Bigl(\beta_{p, X}\wedge K-\frac {1}{K}\Bigr)\Bigl(\mathbb E\Bigl\|\sum_{n=1}^NM^n_1 \phi_{n}(\sigma_1,\ldots,\sigma_{n-1})\Bigr\|  \Bigr)^{\frac 1p}\\
 &= \Bigl(\beta_{p, X}\wedge K-\frac {1}{K}\Bigr) (\mathbb E \|M_{\infty}\|^p)^{\frac 1p},
\end{align*}
where $(i)$ and $(iii)$ follow from the triangle inequality, and $(ii)$ and $(iv)$ follow from \eqref{eq:approxinthm:weakdiffsubproccont}. Hence if $X$ is not UMD, then such $c_{p,X}$ from \eqref{eq:weakdiffsubprocmaingencont} does not exist since $\bigl(\beta_{p, X}\wedge K-\frac {1}{K}\bigr) \to \infty$ as $K\to \infty$. If $X$ is UMD, then such $c_{p,X}$ could exist, and if this is the case, then 
$$
c_{p,X}\geq \lim_{K\to \infty}\Bigl(\beta_{p, X}\wedge K-\frac {1}{K}\Bigr) =\beta_{p, X}.
$$
\end{proof}

\begin{customlemma}{\ref{lemma:weakdiffsubdecomp}}
  Let $X$ be a Banach space, $M^c,N^c:\mathbb R_+ \times \Omega \to X$ be continuous martingales, $M^d, N^d:\mathbb R_+ \times \Omega \to X$ be purely discontinuous martingales, $M^c_0 = N^c_0=0$. Let $M := M^c +M^d$, $N:= N^c + N^d$. Suppose that $N$ is weakly differentially subordinated to $M$. Then $N^c$ is weakly differentially subordinated to $M^c$, and $N^d$ is weakly differentially subordinated to $M^d$.
\end{customlemma}

\begin{proof}
 First notice that a.s. 
 \begin{align*}
  \|N^c_0\|=0&\leq 0=\|M^c_0\|,\\
  \|N^d_0\|=\|N_0\|&\leq \|M_0\|=\|M^d_0\|.
 \end{align*}
Now fix $x^* \in X^*$. It is enough now to prove that $\langle N^c, x^* \rangle$ is differentially subordinated to $\langle M^c, x^*\rangle$, and that $\langle N^d, x^* \rangle$ is weakly differentially subordinated to $\langle M^d, x^*\rangle$. But this follows from \cite[Lemma 1]{Wang}, Remark \ref{rem:YoeMeyR} and the fact that $\langle M^d, x^*\rangle$ and $\langle N^d, x^*\rangle$ are purely discontinuous processes, and $\langle M^c, x^*\rangle$ and $\langle N^c, x^*\rangle$ are continuous processes.
\end{proof}

\end{document}